\documentclass[final]{siamltex1213}
\usepackage{euscript,graphicx,epsfig,amsfonts,mathtools,epstopdf,amssymb,mathrsfs,amsmath}
\usepackage{algorithm,algorithmic}
\usepackage{nameref}
\usepackage{amssymb}

\usepackage{multirow} \usepackage{subfig} \usepackage{tikz}
\usepackage{morefloats} \usepackage{diagbox}

\usepackage{color} 

\newtheorem{remark}[theorem]{Remark}

\newcommand{\R}{\mathbb{R}}

\newcommand{\norm}[1]{\|#1 \|}

\newcommand{\ran}{{\cal R}}

\newcommand{\normtwo}[1]{\lVert #1 \rVert_2}

\newcommand{\argmin}{\arg\min}

\DeclareMathOperator{\Span}{Span} 
\newcommand{\bm}[1]{{#1}}
\newcommand{\mc}[1]{\mathcal{#1}}

\def\cvd{~\vbox{\hrule\hbox{%
  \vrule height1.3ex\hskip0.8ex\vrule}\hrule } }

\hypersetup{ colorlinks=true, linkcolor=blue, citecolor=red}

\title{Multipreconditioned GMRES for Shifted
Systems\footnote{T\lowercase{his version dated} \today}}
\author{\normalsize Tania Bakhos\thanks{Huang Engineering
Center, Stanford University, Stanford, CA 94305 ({\tt
taniab@stanford.edu}, {\tt peterk@stanford.edu}).} \and Peter K.
Kitanidis\footnotemark[2] \and Scott Ladenheim\thanks{School of
Computer Science, University of Manchester, Manchester, UK, M13
9PL ({\tt scott.ladenheim@manchester.ac.uk}).} \and
Arvind~K.~Saibaba\thanks{Department of Mathematics, North
Carolina State University, Raleigh, NC 27695 ({\tt
asaibab@ncsu.edu}).}~\and~Daniel~B.~Szyld\thanks{Department of
Mathematics, Temple University, 1805 N Broad Street,
Philadelphia, PA 19122 ({\tt szyld@temple.edu}). This research
is supported in part by the U.S.\ National Science Foundation
under grant DMS-1418882.}}

\begin{document} \maketitle \begin{abstract} \noindent An
implementation of GMRES with multiple preconditioners (MPGMRES)
is proposed for solving shifted linear systems with
shift-and-invert preconditioners. With this type of
preconditioner, the Krylov subspace can be built without requiring the matrix-vector product with the shifted matrix. Furthermore, the multipreconditioned search space is shown to
grow only linearly with the number of preconditioners. This allows for a more efficient
implementation of the algorithm. The proposed implementation is
tested on shifted systems that arise in computational hydrology
and the evaluation of different matrix functions. The numerical
results indicate the effectiveness of the proposed approach.
\end{abstract}

\section{Introduction} We consider the solution of shifted
linear systems of the form \begin{gather} (A + \sigma_{j}
I)\bm{x}_{j} = \bm{b}, \quad j=1,\ldots,n_{\sigma},
\label{eq:General} \end{gather} where $A\in\R^{n\times n}$ is
nonsingular, $I$ is the $n\times n$ identity matrix, and
$n_{\sigma}$ denotes the number of (possibly complex) shifts
$\sigma_{j}$. We assume that the systems of equations~\eqref{eq:General} have unique solutions for $j=1,\dots,n_\sigma$; i.e., for each $\sigma_j$ we assume $(A+\sigma_j I)$ is invertible. These types of linear systems arise in a wide
range of applications, for instance, in quantum chromodynamics
\cite{frommer1995many}, hydraulic tomography
\cite{saibaba2013flexible}, and in the evaluation of matrix
functions based on the Cauchy integral formula
\cite{higham2010computing}. Computing the solution to these
large and sparse shifted systems remains a significant
computational challenge in these applications.

For large systems, e.g, arising in the discretization of 
three-dimensional partial differential equations, 
solving these systems with direct methods, such as with
sparse LU or Cholesky factorizations, is impractical,
especially considering that a new factorization must be 
performed for each shift.
An attractive option is the use of Krylov subspace iterative
methods. These methods are well-suited for the solution of
shifted systems because they are shift-invariant; see, e.g.,
\cite{Simoncini.Szyld.07}. As a result of this property only a
single shift-independent Krylov basis needs to be generated from 
which all shifted solutions can be computed. 
In this paper 
we consider for its solution a variant of the
generalized minimal residual method (GMRES)
\cite{saad1986gmres},
since the matrix $A$ in \eqref{eq:General} is possibly
nonsymmetric. 

Preconditioning is essential to obtain fast convergence in a
Krylov subspace method. It transforms the original linear system into an equivalent 
system with favorable properties so that the iterative method converges faster. 
For shifted systems, preconditioning can be problematic because it may not
preserve the shift-invariant property of Krylov subspaces. There are a few exceptions though, namely, polynomial preconditioners
\cite{Ahmad.Szyld.VanGijzen.12,jegerlehner1996krylov},
shift-and-invert preconditioners
\cite{gu2007flexible,meerbergen2003solution,meerbergen2010lanczos,saibaba2013flexible},
and nested Krylov approaches \cite{baumann2014nested}. Here, we
consider using several shift-and-invert preconditioners of the
form 
\begin{equation} P_{j}^{-1}=(A+\tau_{j} I)^{-1},
\label{eq:ShiftandInvert}
\end{equation} 
where the values
$\{\tau_{j}\}_{j=1}^{n_p}$ correspond to the different shifts,
with $n_p \ll n_{\sigma}$. For a generic shift $\tau$, we denote
the preconditioner $P_{\tau}^{-1}$. One advantage of using the preconditioner in~\eqref{eq:ShiftandInvert} is that, as shown in the next section, the appropriate preconditioned Krylov subspace can be built without performing the matrix-vector product with $(A+\sigma_j I)$. 

The reason to consider several shifted preconditioners is that a single preconditioner alone is insufficient to effectively
precondition across all shifts, as observed in
\cite{gu2007flexible,saibaba2013flexible}. Incorporating more than one
preconditioner necessitates the use of flexible Krylov subspace
methods; see, e.g.,
\cite{saad1993flexible,Szyld.Vogel.01,vogel2007flexible}.
Flexible methods allow for the application of a different
preconditioner at each iteration and  the
preconditioners are cycled in some pre-specified order. However, this means that information from only one preconditioner is
added to the basis at each iteration. In addition, the order in
which the preconditioners are applied can have a significant
effect on the convergence of the iterative method. To
effectively precondition across the range of shifts we use 
multipreconditioned GMRES (MPGMRES) \cite{greifmpgmres}.

Multipreconditioned Krylov subspace methods offer the capability of using several
preconditioners at each iteration during the course of solving linear systems. 
In contrast to flexible methods using one preconditioner at each iteration,
the MPGMRES approach uses information from all the preconditioners
in every iteration. Since MPGMRES builds a larger and richer search space, the
convergence is expected to be faster than FGMRES. However, the cost per
iteration for MPGMRES can be substantial. To deal with these large computational costs, 
and in particular the exponential growth of the search space, a selective version of the algorithm is proposed in~\cite{greifmpgmres}, where the growth of the search space is linear per iteration, i.e., the same growth as a block method of block size $n_p$.

Our major contribution is the development of a new iterative solver to handle shifted systems of the
form~\eqref{eq:General} that uses multiple preconditioners together with the fact that no multiplication with $(A+\sigma_j I)$ is needed. 
Our proposed method is motivated by MPGMRES and builds a search 
space using information at each iteration 
from multiple shift-and-invert preconditioners. By searching for optimal 
solutions over a richer space, we anticipate the convergence to
be rapid and the number of iterations to be low.  
For this class of preconditioners, the resulting Krylov space
has a special form, and we show that the search space 
grows only linearly. This yields a method with contained
computational and storage costs
per iteration. Numerical experiments illustrate that the proposed
solver is more effective than standard Krylov
methods for shifted systems, in terms of both iteration counts
and overall execution time. 

We show that the proposed approach can also be
applied to the solution of more general shifted systems of the form
$(K+\sigma M)\bm{x} = \bm{b}$ with shift-and-invert
preconditioners of the form $(K+\tau
M)^{-1}$; see \ Sections~\ref{sec:ExtensiontoK}
and~\ref{sec:Hydrology}. 

The paper is structured as follows. In Section
\ref{sec:MPGMRES}, we briefly review some basic properties of GMRES for
the solution of shifted systems with shift-and-invert
preconditioners. We also review the MPGMRES algorithm, and
a flexible GMRES (FGMRES) algorithm for shifted systems proposed in
\cite{saibaba2013flexible}. In Section \ref{sec:shiftedMPGMRES}, we present the proposed MPGMRES implementation for shifted
systems with shift-and-invert preconditioners, 
a new theorem characterizing the linear growth of the MPGMRES
search space, and discuss an efficient implementation of the
algorithm that exploits this linear growth. In Section
\ref{sec:Hydrology}, we present numerical experiments for
shifted linear systems arising in hydraulic tomography
computations. In Section \ref{sec:MatFun}, we present another
set of numerical experiments for the solution of shifted systems 
arising from the evaluation of different matrix functions.
Concluding remarks are given in Section~\ref{sec:Conc}.

 \section{GMRES, MPGMRES, and FMGRES-Sh}
\label{sec:MPGMRES}
To introduce the proposed multipreconditioned approach for GMRES applied to shifted linear systems, the original GMRES algorithm \cite{saad1986gmres}, the multipreconditioned version MPGMRES \cite{greifmpgmres}, GMRES for shifted systems, and flexible GMRES for shifted systems (FGMRES-Sh)~\cite{gu2007flexible,saibaba2013flexible} are first reviewed.

\subsection{GMRES}
GMRES is a Krylov subspace iterative method for solving large, sparse, nonsymmetric linear systems of the form
\begin{equation}
A\bm{x}=\bm{b}, \quad A\in\R^{n\times n},\quad \bm{b}\in\R^{n}.
\label{eq:LinearSystem}
\end{equation}
From an initial vector $\bm{x}_{0}$ and corresponding residual $\bm{r}_{0}=\bm{b}-A\bm{x}_{0}$, GMRES computes at step $m$ an approximate solution $\bm{x}_{m}$ to \eqref{eq:LinearSystem} belonging to the affine Krylov subspace $\bm{x}_{0}+\mc{K}_{m}(A,\bm{r}_{0})$, where 
\begin{equation}\label{eqn:krylov} \nonumber
\mc{K}_{m}(A,\bm{r}_{0}) \equiv \>\Span \{\bm{r}_{0},A\bm{r}_{0},\ldots,A^{m-1}\bm{r}_{0} \}.
\end{equation}
The corresponding residual $\bm{r}_{m}=\bm{b}-A\bm{x}_{m}$ is characterized by the minimal residual condition
\begin{equation*}
\norm{\bm{r}_{m}}_2 = \min_{\bm{x}\in\bm{x}_{0} + \mc{K}_{m}(A,\bm{r}_{0})} \norm{\bm{b}-A\bm{x}}_2.
\end{equation*}

The approximate solution is the sum of the vector $x_0$ and a linear combination of the orthonormal basis vectors of the Krylov subspace, 
which are generated by the Arnoldi algorithm. The first vector in the Arnoldi algorithm is the normalized initial residual $\bm{v}_{1}= \bm{r}_{0}/\beta$, where $\beta=\norm{\bm{r}_{0}}_2$, and subsequent vectors $\bm{v}_{k}$ are formed by orthogonalizing $A\bm{v}_{k-1}$ against all previous basis vectors. Collecting the basis vectors into the matrix
$V_m = [ v_1, \ldots, v_m]$, one can write the Arnoldi relation
\begin{equation}
AV_{m}=V_{m+1}\bar{H}_{m},
\label{eq:Arnoldi}
\end{equation}
where 
$\bar{H}_m\in\R^{m+1\times m}$ is an upper Hessenberg matrix whose entries are the orthogonalization coefficients.

Thus, solutions of the form $\bm{x}=\bm{x}_{0} + V_{m}\bm{y}_{m}$ are sought for some  $\bm{y}_{m}\in\R^{m}$. Using the Arnoldi relation, and the fact that $V_{m}$ is orthonormal, the GMRES minimization problem can be cast as the equivalent, smaller least-squares minimization problem
\begin{equation}
\norm{\bm{r}_{m}}_{2} = \min_{\bm{y}\in\R^{m}} \norm{ \beta \bm{e}_{1} - \bar{H}_{m}\bm{y}}_{2},
\label{eq:GMRESMin}
\end{equation}
with solution $y_m$, 
where $\bm{e}_{1} = [1,0,\ldots, 0]^{T}$ is the first standard basis vector in $\R^{m}$.

\subsection{MPGMRES}
The multipreconditioned GMRES method allows multiple preconditioners to be incorporated in the solution of a given linear system.\ The method differs from standard preconditioned Krylov subspace methods in that the approximate solution is found in a larger and richer search space~\cite{greifmpgmres}. This multipreconditioned search space grows at each iteration by applying each of the preconditioners to all current search directions. 
In building this search space some properties of both flexible and block Krylov subspace methods are used.
As with GMRES,
the MPGMRES algorithm then produces an approximate solution 
satisfying a minimal residual condition over this larger subspace. 

Computing the multi-Krylov basis is similar to computing the Arnoldi basis in GMRES but with a few key differences. Starting from an initial vector $\bm{x}_{0}$, corresponding initial residual $\bm{r}_0$, and setting $V^{(1)}=\bm{r}_{0}/\beta$, the first MPGMRES iterate $\bm{x}_{1}$ is found in 
\begin{equation*}
\bm{x}_{1} \in \bm{x}_{0} + \Span\{ P_{1}^{-1}\bm{r}_{0},\ldots,P_{n_p}^{-1}\bm{r}_{0}\},
\end{equation*}
such that the corresponding residual has minimum 2-norm over all vectors of this form. 
Equivalently, $\bm{x}_{1}=\bm{x}_{0}+Z^{(1)}\bm{y}_{1}$, where $Z^{(1)}=[P^{-1}_{1}\bm{r}_{0}\cdots P_{n_p}^{-1}\bm{r}_{0}]\in\R^{n\times n_p}$ 
and the vector $\bm{y}_{1}$ minimizes the residual. Note that the corresponding residual belongs to the space
\begin{equation*}
\bm{r}_{1} \in \bm{r}_{0}+ \Span\{AP_{1}^{-1}\bm{r}_{0},\ldots, AP_{n_p}^{-1}\bm{r}_{0}\}.
\end{equation*}
In other words, at the first iteration, the residual can be expressed as a first degree multivariate (non-commuting) matrix polynomial with arguments $AP_{j}^{-1}$ applied to the initial residual, i.e.,
\begin{equation*}
\bm{r}_{1}= \bm{r}_{0} + \sum_{j=1}^{n_p} \alpha_{j}^{(1)}AP_{j}^{-1}\bm{r}_{0}  \equiv { q_{1}(AP_{1}^{-1},\ldots,AP_{n_p}^{-1})}\bm{r}_{0},
\end{equation*} 
with the property that $q_{1}(0,\ldots,0) = 1$.

The next set of basis vectors $V^{(2)}$, with orthonormal columns, are generated by orthogonalizing the columns of $AZ^{(1)}$ against $V^{(1)}=\bm{r}_{0}/\beta$, and performing a thin QR factorization. The orthogonalization coefficients generated in this process are stored in matrices $H^{(j,1)}$, $j=1,2$. The space of search directions is increased by applying each preconditioner to this new matrix; i.e., $Z^{(2)}=[P_{1}^{-1}V^{(2)}\ldots P_{n_p}^{-1}V^{(2)}]\in\R^{n\times n_p^{2}}$. The general procedure at step $k$ is to block orthogonalize $AZ^{(k)}$ against the matrices $V^{(1)},\ldots,V^{(k)}$, and then performing a thin QR factorization to generate $V^{(k+1)}$ with orthonormal columns. By construction, the number of columns in $Z^{(k)}$, called search directions, and $V^{(k)}$, called basis vectors, is $n_p^{k}$.


Gathering all of the matrices, $V^{(k)}$, $Z^{(k)}$, and $H^{(j,i)}$ generated in the multipreconditioned Arnoldi procedure into larger matrices yields a decomposition of the form
\begin{equation}
A\mc{Z}_{m} = \mc{V}_{m+1}\bar{\mc{H}}_{m},
\label{eq:MPArnodliRelation}
\end{equation}
where
\begin{equation*}
\mc{Z}_{m}=
\begin{bmatrix}
Z^{(1)} & \cdots & Z^{(m)} \\
\end{bmatrix}
,
\qquad
\mc{V}_{m+1}
=
\begin{bmatrix}
V^{(1)} & \cdots & V^{(m+1)} \\
\end{bmatrix},
\end{equation*}
and
\begin{equation*}
\bar{\mc{H}}_{m}
=
\begin{bmatrix}
H^{(1,1)} & H^{(1,2)} & \cdots & H^{(1,m)} \\
H^{(2,1)} & H^{(2,2)} & & H^{(2,m)} \\
	& \ddots & &\\
	& & H^{(m,m-1)} & H^{(m,m)} \\
	& & & H^{(m+1,m)} \\
\end{bmatrix}\cdot
\end{equation*}
Similar to the standard Arnoldi method, the matrix $\bar{\mc{H}}_{m}$ is upper Hessenberg.  Introducing the constant
\begin{equation*}
\Sigma_{m} = \sum_{k=0}^{m} n_p^{k} = \frac{n_p^{m+1}-1}{n_p-1},
\end{equation*}
we have that
\begin{equation*}
\mc{Z}_{m}\in\R^{n\times (\Sigma_{m}-1)}, \quad \mc{V}_{m+1}\in\R^{n \times \Sigma_{m}}, \quad \bar{\mc{H}}_{m}\in\R^{\Sigma_{m}\times (\Sigma_{m}-1)}. 
\end{equation*}

The multipreconditioned search space is spanned by the columns of $\mc{Z}_{m}$ so that approximate 
solutions have the form $\bm{x}_{m} = \bm{x}_{0} + \mc{Z}_{m}\bm{y}$ for $\bm{y}\in \R^{\Sigma_{m}-1}$. 
Thus, thanks to \eqref{eq:MPArnodliRelation}, 
MPGMRES computes $y$ as the solution of the least squares problem
\begin{equation*}
\norm{\bm{r}_{m}}_{2} = \min_{\bm{y}\in\R^{\Sigma_{m}-1}}\Vert \beta\bm{e}_{1} - \bar{\mc{H}}_{m}\bm{y}\Vert_{2}~.
\end{equation*}

This is analogous to the GMRES minimization problem, except that here, a larger least squares problem is solved. Note that both $\mc{Z}_{m}$ and $\mc{V}_{m}$ are stored. The complete MPGMRES method consists of the multipreconditioned Arnoldi procedure coupled 
with the above least squares minimization problem. Furthermore, we highlight the expression of the MPGMRES residual as a 
multivariate polynomial, i.e., we can 
write $\bm{r}_{m}= q_{m}(AP_{1}^{-1},\ldots,AP_{n_p}^{-1})\bm{r}_{0}$, where $q_{m}(X_{1},\ldots,X_{n_p})\in\mathbb{P}_{m}$. 
Here $\mathbb{P}_{m}~\equiv~\mathbb{P}_{m}[X_1,\dots,X_{n_p}]$ is the space of non-commuting polynomials of degree $m$ in $n_p$ variables such that $q_{m}(0,\ldots,0)=1$.


In the above description of MPGMRES we have tacitly assumed that the matrix $\mc{Z}_{m}$ is of full rank. However, in creating the multipreconditioned search space, columns of $\mc{Z}_{m}$ may become linearly dependent. Strategies for detecting when such linear dependencies arise and then deflating the search space accordingly have been proposed. For the implementation details of the complete MPGMRES method, as well as information on selective variants for maintaining linear growth of the search space, we refer the reader to \cite{greifmpgmres,Greif.Rees.Szyld.14a}. 

\subsection{GMRES for shifted systems}
The solution of shifted systems using the GMRES method requires minor but important modifications. One key idea is to exploit the shift-invariant property of Krylov subspaces, i.e., $\mc{K}_{m}(A + \sigma I,\bm{b}) =\mc{K}_{m}(A,\bm{b})$, and generate a single approximation space from which all shifted solutions are computed. For shifted systems, the Arnoldi relation is
\begin{equation}
(A +\sigma I)V_{m} =  V_{m+1}\left(\bar{H}_{m} + \sigma\begin{bmatrix} I_m \\ 0 \end{bmatrix}\right) \equiv V_{m+1}\bar{H}_{m}(\sigma),
\label{eq:ShiftedArnoldi}
\end{equation}
where $I_{m}$ is the $m\times m$ identity matrix. The matrices $V_{m}$, $\bar{H}_{m}$ are the same as in \eqref{eq:Arnoldi} and are independent of the shift $\sigma$. 

Using (\ref{eq:ShiftedArnoldi}),
the equivalent shift-dependent GMRES minimization problem is
\begin{equation*}
\norm{ \bm{r}_{m}(\sigma)}_{2} = \min_{y\in\R^{m}} \norm{ \beta e_{1} - \bar{H}_{m}(\sigma)\bm{y}}_{2}~.
\end{equation*}
This smaller least squares problem is solved for each shift. Note that the computationally intensive step of generating the basis vectors $V_{m}$ is performed only once and the cheaper solution of the projected, smaller minimization problem occurs for each of the shifts. We remark that the GMRES initial vector is $\bm{x}_{0}=0$ for shifted systems since the initial residual must be shift independent. 

\subsection{FGMRES for shifted systems}
As previously mentioned, we consider using several shift-and-invert preconditioners in the solution of shifted systems. This is due to the fact that a single shift-and-invert preconditioner is ineffective for preconditioning over a large range of shift values $\sigma$, a key observation made in \cite{gu2007flexible,saibaba2013flexible}. Using FMGRES, one can cycle through and apply $P_{j}^{-1}$ for each value of $\tau_{j}$. 
Incorporating information from the different shifts into the approximation space improves convergence compared to GMRES with a single shifted preconditioner.

To build this approximation space the authors in \cite{gu2007flexible} use the fact that
\begin{equation}\label{eqn:ainv}
(A+\sigma I)(A + \tau I)^{-1} = I + (\sigma - \tau) (A + \tau I)^{-1},
\end{equation}
from which it follows that
\begin{equation}
\mc{K}_{m}((A+\sigma I)P_{\tau}^{-1},\bm{v})=\mc{K}_{m}(P_{\tau}^{-1},\bm{v}).
\label{eq:flexibleshift}
\end{equation}
Therefore, they propose building a Krylov subspace based on $P_{\tau}^{-1}$. Note that the Krylov subspace $\mc{K}_{m}(P_{\tau}^{-1},\bm{v})$ is independent of $\sigma$ and 
therefore each shifted system can be projected onto this approximation space. The flexible approach constructs a basis $Z_{m}$ of the approximation space, where each column $\bm{z}_{j}= P_{j}^{-1}\bm{v}_{j}$ corresponds to a different shift. This gives the following flexible, shifted Arnoldi relation
\begin{equation*}
(A+\sigma I)Z_{m} = (A+\sigma I)[P_{1}^{-1}\bm{v}_{1}\cdots P_{m}^{-1}\bm{v}_{m}] = V_{m+1}\left(\bar{H}_{m}(\sigma I_{m} - T_{m}) + \begin{bmatrix} I_m \\ 0 \end{bmatrix}\right),
\end{equation*}
where $T_{m}= \text{diag}(\tau_{1},\ldots,\tau_{m})$.

Although this approach is more effective than a single preconditioner, information from only a single shift is incorporated into the search space at each iteration and the order in which the preconditioners are applied can affect the performance of FGMRES. These potential deficiencies motivate the proposed multipreconditioned algorithm, which allows for information from all preconditioners (i.e., all shifts $\tau_1,\dots,\tau_m$) to be built into the approximation space at every iteration.
 \section{MPGMRES for shifted systems}
\label{sec:shiftedMPGMRES}

In this section, we present a modification of the MPGMRES algorithm to handle shifted systems with multiple shift-and-invert preconditioners,
where the relation~\eqref{eq:flexibleshift} plays a crucial role. 
The new algorithm is referred to as MPGMRES-Sh. We shall prove that the growth of the search space at each iteration is linear in the number of preconditioners thereby leading to an efficient algorithm. This is in contrast to the original MPGMRES algorithm where the dimension of the search search space grows exponentially in the number of preconditioners.   

The proposed implementation of MPGMRES for solving shifted systems \eqref{eq:General} with preconditioners \eqref{eq:ShiftandInvert} adapts the flexible strategy of \cite{saibaba2013flexible}, and using relation~\eqref{eq:flexibleshift} builds a shift-invariant multipreconditioned search space that is used to solve for each shift. 
We assume throughout that $(A+\tau_j I) $ is invertible for all $j=1,\dots,n_p$ so that the preconditioners $P_j^{-1} = (A+\tau_j I)^{-1}$ are well-defined.  
{The MPGMRES algorithm proceeds by applying all $n_p$ preconditioners to the columns of $V^{(k)}$, which results in}
\begin{equation} \label{eqn:Zk}
Z^{(k)} = [(A+\tau_1I)^{-1}V^{(k)},\ldots, (A+\tau_{n_p} I)^{-1}V^{(k)}] \in \mathbb{R}^{n\times n_p^{k} }, \quad V^{(k)} \in \mathbb{R}^{n\times n_p^{k-1}}.
\end{equation}
Rearranging this we obtain the relation
\begin{equation} 
AZ^{(k)} + Z^{(k)}T^{(k)}  =   V^{(k)}E^{(k)},   \quad E^{(k)} = \bm{e}_{n_{p}}^T \otimes I_{n_p^{k-1}},
\label{eqn:completempd}
\end{equation}
where $\bm{e}_{n_p}^{T} = {\underbrace{[1,\ldots,1]}_{n_p\ \text{times}}}^{T}$ and
\begin{equation}
T^{(k)} =   \text{block diag} \left\{ \underbrace{\tau_1,\dots,\tau_1}_{n_p^{k-1} \text{times} },\ldots, \underbrace{\tau_{n_p},\dots,\tau_{n_p}}_{n_p^{k-1} \text{times} }\right\}\in\R^{n_{p}^{k}\times n_{p}^{k}}. 
\label{eqn:shiftmatT}
\end{equation}
Concatenating the matrices $Z^{(k)}$ and $V^{(k)}$ for $ k=1,\dots,m$,
into $\mc{Z}_{m}$ and $\mc{V}_{m}$, we rewrite the $m$ equations in \eqref{eqn:completempd} into a matrix relation 
\begin{equation}
A\mathcal{Z}_{m} +\mathcal{Z}_mT_m   =   \mathcal{V}_mE_m,  
\label{eqn:ArnoldiRelation}
\end{equation}
where $T_m  =\text{block diag}\{ T^{(1)} ,\ldots,T^{(m)}\}$ and $E_m = \text{block diag}\left\{ E^{(1)},\ldots,E^{(m)}\right\}$.

To generate the next set of vectors $V^{(k+1)}$ we use a block Arnoldi relationship of the form
\[ V^{(k+1)}H^{(k+1,k)} = Z^{(k)} - \sum_{j=1}^k V^{(j)}H^{(j,k)},\]
in combination with~\eqref{eqn:Zk}, so that at the end of $m$ iterations, the flexible multipreconditioned Arnoldi relationship holds, i.e., $\mathcal{Z}_m = \mathcal{V}_{m+1}\bar{\mc{H}}_m$. In summary, we have the system of equations 
\begin{align}
\mathcal{Z}_m =  & \quad \mathcal{V}_{m+1}\bar{\mc{H}}_m  \label{eq:Arnoldi1},\\ 
\mathcal{Z}_mT_m =& \quad  \mathcal{V}_{m+1}\bar{\mc{H}}_m T_m, \label{eq:Arnoldi2}\\
A\mathcal{Z}_m + \mathcal{Z}_mT_m= & \quad \mathcal{V}_mE_m.  \label{eq:Arnoldi3}
\end{align}

\begin{remark}
From~\eqref{eqn:ainv}, it follows that $\text{Span}\{\mathcal{Z}_m\} =\text{Span}\{ (A+\sigma I)\mathcal{Z}_m\} $; see also~\eqref{eq:flexibleshift}. As a consequence, in Step~\ref{item:Zk} of Algorithm~\ref{alg:mpgmressh} we do not need to explicitly compute matrix-vector products with $A$. 
\end{remark}

Combining relations \eqref{eq:Arnoldi1}--\eqref{eq:Arnoldi3}, we obtain the following flexible 
multipreconditioned Arnoldi relationship for each shift $\sigma$:
\begin{equation}
(A + \sigma I)\mathcal{Z}_m = \mathcal{V}_{m+1} {\left( \begin{bmatrix} E_m \\ 0 \end{bmatrix} + \bar{\mc{H}}_m(\sigma I - T_m) \right) } \equiv{ \mc{V}_{m+1}\bar{\mc{H}}_m(\sigma;T_m)}.
\label{eqn:ArnoldiIdentity}
\end{equation}
Finally, searching for approximate solutions of the form $\bm{x}_m = \mathcal{Z}_m\bm{y}_m$ (recall that $x_0 = 0$) and using the minimum residual condition we have the following minimization problem for each shift:
\begin{align}
\nonumber
\Vert \bm{r}_{m}(\sigma) \Vert_{2} &= \min_{\bm{x}\in \Span\{\mc{Z}_{m}\} } \Vert \bm{b}-(A+\sigma I)\bm{x}_{m} \Vert _{2} 
 = \min_{\bm{y}\in \R^{\Sigma_m-1}}\Vert \bm{b} - (A+\sigma I)\mc{Z}_{m}\bm{y} \Vert_{2} \\
&= \min_{\bm{y}\in \R^{\Sigma_m-1}} \Vert \mc{V}_{m+1}(\beta \bm{e}_{1} - \bar{\mc{H}}_{m}(\sigma;T_{m})\bm{y}) \Vert_{2} 
= \min_{\bm{y}\in \R^{\Sigma_m-1}} \Vert \beta\bm{e}_{1} - \bar{\mc{H}}_{m}(\sigma;T_{m})\bm{y} \Vert_{2}. 
\label{eq:minsigma}
\end{align} 
The application of the multipreconditioned Arnoldi method using each preconditioner $P_{j}^{-1}$ in conjunction with the 
solution of the above minimization problem is the complete MPGMRES-Sh method, see Algorithm \ref{alg:mpgmressh}.

\begin{algorithm}[h!]
\begin{algorithmic}[1]
\REQUIRE 
Matrix $A$, right-hand side $\bm{b}$, preconditioners $\{A +\tau_j I\}_{j=1}^{n_p}$, shifts $\{\sigma_j\}_{j=1}^{n_\sigma}$, $\{\tau_j\}_{j=1}^{n_p}$, and number of iterations $m$.
\STATE Compute $\beta = \norm{\bm{b}}_{2}$ and $V^{(1)} =  \bm{b}/\beta$.
\FOR {$k=1,\dots,m$ }  \label{line:begin}
\STATE $Z^{(k)} = [P_{1}^{-1}V^{(k)},\ldots,P_{n_p}^{-1}V^{(k)}]$  \label{line:end}
\STATE $W= Z^{(k)}$ \label{item:Zk}
\FOR {j = 1,\dots,k}
\STATE $H^{(j,k)} = (V^{(j)})^TW$
\STATE $W = W - V^{(j)}H^{(j,k)}$
\ENDFOR
\STATE $W =  V^{(k+1)}H^{(k+1,k)}$ \  \COMMENT{thin QR factorization}
\ENDFOR
\FOR {$j=1,\dots,n_\sigma$}
\STATE Compute $\bm{y}_m(\sigma_{j}) = \argmin_\bm{y} \norm{\beta \bm{e}_1 - \bar{\mc{H}}_m(\sigma_j;T_m)\bm{y} }$ for each shift. \label{line:min}
\STATE $\bm{x}_{m}(\sigma_{j}) = \mc{Z}_{m}\bm{y}_{m}(\sigma_{j})$, where $\mc{Z}_{m} = [Z^{(1)}, \cdots, Z^{(m)}]$ 
\ENDFOR
\RETURN The approximate solution $x_m(\sigma_j)$ for $j=1,\dots,n_\sigma$. 
\end{algorithmic}
\caption{Complete MPGMRES-Sh}
\label{alg:mpgmressh}
\end{algorithm}


\subsection{Growth of the search space}
It can be readily seen  in Algorithm~\ref{alg:mpgmressh} that the number of columns of $\mc{Z}_m$ and $\mc{V}_m$ grows exponentially. 
However, as we shall prove below, with the use of shift-and-invert preconditioners the dimension of this space grows only linearly. 

Noting that $r_m(\sigma) \in \Span\{\mc{V}_{m+1}\}$, the residuals produced by the MPGMRES-Sh method are of the form
\begin{equation}
\bm{r}_{m} (\sigma) \in q_{m}(P_{1}^{-1},\ldots,P_{n_p}^{-1})\bm{b},
\label{eq:mpresidual}
\end{equation}
for a polynomial $q_{m}\in\mathbb{P}_{m}$. Recall that $\mathbb{P}_{m}$ is the space of multivariate polynomials of degree $m$ in 
$n_p$ (non-commuting) variables such that $q_{m}(0,\ldots,0)=1$. Note that this polynomial is independent of the shifts $\sigma_j$. 

To illustrate this more clearly, for the case of $n_p=2$ preconditioners, the first two residuals are of the form
\begin{align*}
\bm{r}_{1} (\sigma) &\in \bm{b} + \alpha^{(1)}_{1}P_{1}^{-1}\bm{b} + \alpha^{(1)}_{2}P_{2}^{-1}\bm{b}, \\
\bm{r}_{2}(\sigma) &\in \bm{b} + \alpha^{(2)}_{1}P_{1}^{-1}\bm{b} + \alpha^{(2)}_{2}P_{2}^{-1}\bm{b} + \alpha_{3}^{(2)}(P_{1}^{-1})^{2}\bm{b} + \alpha_{4}^{(2)}P_{1}^{-1}P_{2}^{-1}\bm{b} \\ 
& ~~~+ \alpha_{5}^{(2)}P_{2}^{-1}P_{1}^{-1}\bm{b} + \alpha_{6}^{(2)}(P_{2}^{-1})^{2}\bm{b} .
\end{align*}

The crucial point we subsequently prove is that the cross-product terms of the form $P_{i}^{-1}P_{j}^{-1}\bm{v}\in \Span\{P_{i}^{-1}\bm{v},P_{j}^{-1}\bm{v}\}$ for $i\neq j$. In particular, only the terms that are purely powers of the form $P_{i}^{-k}\bm{v}$ need to be computed. This allows $\Span\{\mathcal{Z}_m\}$ to be expressed as the sum of Krylov subspaces generated by 
individual shift-and-invert preconditioners $P_i^{-1}$; cf.\ \eqref{eq:flexibleshift}. 
Thus, the search space has only linear growth in the number of preconditioners. 

For ease of demonstration, we  first prove this for $n_p=2$ preconditioners. 
To prove the theorem we need the following two lemmas. 


\begin{lemma}
Let $P_{1}^{-1},P_{2}^{-1}$ be shift-and-invert preconditioners as defined in \eqref{eq:ShiftandInvert} with $\tau_{1}\neq\tau_{2}$. Then $P_{1}^{-1}P_{2}^{-1}\bm{v}\in\Span\{P_{1}^{-1}\bm{v},P_{2}^{-1}\bm{v}\}$. 
\label{lem:CrossProduct}
\end{lemma}

\emph{Proof.}
We show there exist constants $\gamma_{1},\gamma_{2}$ such that
\begin{equation}
P_{1}^{-1}P_{2}^{-1}\bm{v} = \gamma_{1}P_{1}^{-1}\bm{v} + \gamma_{2}P_{2}^{-1}\bm{v} = P_{2}^{-1}P_{1}^{-1}\bm{v}.
\label{eq:constants}
\end{equation}
Setting $\bm{v}=P_{2}\bm{w}$ and left multiplying by $P_{1}$ gives the following equivalent formulation
\begin{align*}
\bm{w} &= \gamma_{1} P_{2}\bm{w} + \gamma_{2}P_{1}\bm{w}
   = \gamma_{1} (A+\tau_{2}I)\bm{w} + \gamma_{2} (A+\tau_{1}I)\bm{w} \\
   &= (\gamma_{1}+\gamma_{2})A\bm{w} + (\gamma_{1}\tau_{2} + \gamma_{2}\tau_{1})\bm{w}.
\end{align*}
By equating coefficients we obtain that \eqref{eq:constants} holds for
\begin{equation*}
\gamma_{1}= \frac{1}{\tau_{2}-\tau_{1}} = -\gamma_{2}. ~~~\cvd 
\end{equation*}
\begin{remark}
Note that $P_1^{-1}P_2^{-1} = P_2^{-1}P_1^{-1}$, i.e., the preconditioners commute even though we have not assumed that $A$ is diagonalizable. We make use of this observation repeatedly.
\end{remark}
\begin{lemma}
Let $P_{1}^{-1},P_{2}^{-1}$ be shift-and-invert preconditioners as defined in \eqref{eq:ShiftandInvert} with $\tau_{1}\neq\tau_{2}$, then \[P_{1}^{-m}P_{2}^{-n}\bm{v} \in \mc{K}_{m}(P_{1}^{-1},P_{1}^{-1}\bm{v}) + \mc{K}_{n}(P_{2}^{-1},P_{2}^{-1}\bm{v}).\]
\label{lem:CrossProductMN}
\end{lemma}
\emph{Proof.}
We proceed by induction. The base case when $m=n=1$ is true by Lemma~\ref{lem:CrossProduct}. Assume that the induction hypothesis is true 
for 
$ 1 \leq j \leq m$, 
$ 1 \leq k \leq n$, 
that is, we assume there exist coefficients such that
\begin{equation*}
P_{1}^{-m}P_{2}^{-n}\bm{v} = \sum_{j=1}^{m}\alpha_{j}P_{1}^{-j}\bm{v} + \sum_{j=1}^{n}\alpha_{j}'P_{2}^{-j}\bm{v}.
\end{equation*}
Now consider
\begin{align*}
 P_{1}^{-(m+1)}P_{2}^{-(n+1)}\bm{v} &= (P_{1}^{-1}P_{2}^{-1})(P_{1}^{-m}P_{2}^{-n}\bm{v}) \\
&= \left(\gamma_1 P_1^{-1} + \gamma_2 P_2^{-1} \right) \left(\sum_{j=1}^{m}\alpha_{j}P_{1}^{-j}\bm{v} + \sum_{j=1}^{n}\alpha_{j}'P_{2}^{-j}\bm{v}\right) \\ 
&= \sum_{j=2}^{m+1}\gamma_1{\alpha}_{j-1} P_{1}^{-j}\bm{v} + \sum_{j=2}^{n+1}\alpha_{j-1}'\gamma_2P_{2}^{-j}\bm{v} \\
& \qquad 
+  \sum_{j=1}^{m}\alpha_{j}\gamma_2P_2^{-1}P_{1}^{-j}\bm{v} + \sum_{j=1}^{n}\gamma_1\alpha_{j}'P_1^{-1}P_{2}^{-j}\bm{v} \\
&= \sum_{j=1}^{m+1} \tilde{\alpha}_{j} P_{1}^{-1}\bm{v} + \sum_{j=1}^{n+1} \tilde{\alpha}_{j}' P_{2}^{-1}\bm{v}. \\
\end{align*}

\vspace*{-6mm}

The first equality follows from the commutativity of $P_{1}^{-1}$ and $P_{2}^{-1}$. 
The second equality is the induction hypothesis. 
The third equality is just a result of the distributive property.  The final equality results from applications of Lemma \ref{lem:CrossProduct} and then expanding and gathering like terms. 
It can be readily verified that every term in this expression for $P_{1}^{-(m+1)}P_{2}^{-(n+1)}\bm{v} $ belongs to $\mc{K}_{m+1}(P_{1}^{-1},P_{1}^{-1}\bm{v}) + \mc{K}_{n+1}(P_{2}^{-1},P_{2}^{-1}\bm{v})$.~\cvd
\begin{theorem}
Let $P_{1}^{-1},P_{2}^{-1}$ be shift-and-invert preconditioners as defined in \eqref{eq:ShiftandInvert} with $\tau_{1}\neq\tau_{2}$. At the $m^{\text{th}}$ step of the \text{MPGMRES-Sh} algorithm we have
\begin{equation}
\bm{r}_{m} (\sigma) =  q^{(1)}_{m}(P_{1}^{-1})\bm{b} + q^{(2)}_{m}(P_{2}^{-1})\bm{b},
\end{equation}
where $q^{(1)}_{m},q^{(2)}_{m}\in\mathbb{P}_{m}[X]$ are such that $q^{(1)}_{m}(0) + q^{(2)}_{m}(0)=1$. 
That is, the residual is expressed as the sum of two (single-variate) polynomials of degree $m$
on the appropriate preconditioner, applied to $\bm{b}$. 
\label{thm:LinearGrowth}
\end{theorem}

\begin{proof}
Recall that at step $m$ of the MPGMRES-Sh algorithm the residual can be expressed as the multivariate polynomial (cf.\ \eqref{eq:mpresidual})
\begin{equation} \label{rmsigma}
\bm{r}_{m} (\sigma)=  q_{m}(P_{1}^{-1},P_{2}^{-1})\bm{b},
\end{equation}
with $q_{m}\in\mathbb{P}_{m}[X_1,X_2]$ and $q_{m}(0,0)=1$. Thus, we need only show that the multi-variate polynomial $q_{m}$ can be expressed as the sum of two polynomials in $P_{1}^{-1}$ and $P_{2}^{-1}$. 
By Lemma \ref{lem:CrossProductMN}, any cross-product term involving $P_{1}^{-j}P_{2}^{-\ell}$ can be expressed as a linear combination of powers of only $P_{1}^{-1}$ or $P_{2}^{-1}$. 
Therefore, gathering like terms we can express \eqref{rmsigma} as
\begin{align*}
\bm{r}_m (\sigma) &= \> \sum_{j=1}^{m}\alpha_{j}P_{1}^{-j}\bm{b} + \sum_{\ell=1}^{m}\alpha_{\ell}'P_{2}^{-\ell}\bm{b} 
=\>  q^{(1)}_{m}(P_{1}^{-1})\bm{b} + q^{(2)}_{m}(P_{2}^{-1})\bm{b} ,
\end{align*}
where $q_{m}^{(i)}\in\mathbb{P}_{m}[X]$ and $q_{m}^{(1)}(0) + q_{m}^{(2)}(0) = 1$.
\end{proof}

For the general case of $n_p >1$, we have the same result, as stated below. Its proof is given in Appendix~\ref{appendixA}.
\begin{theorem} \label{gen.theo}
Let $\{P_{j}^{-1}\}_{j=1}^{n_p}$ be shift-and-invert
preconditioners as defined in \eqref{eq:ShiftandInvert}, where
$\tau_{j}\neq\tau_{i}$ for $j\neq i$. At the $m^{\text{th}}$
step of the MPGMRES-Sh algorithm we have
\begin{equation*}
\bm{r}_{m}(\sigma) =  \sum_{j=1}^{n_p} q_{m}^{(j)}(P_{j}^{-1})\bm{b},
\end{equation*}
where $q_{m}^{(j)}\in \mathbb{P}_m[X]$ and
$\sum_{j=1}^{n_p}q_{m}^{(j)}(0)=1$. In other words, the residual
is a sum of $n_p$ single-variate polynomials of degree $m$ on the appropriate preconditioner, applied to~$\bm{b}$. 
\end{theorem}

\begin{remark}\label{r_main}
The residual can be equivalently expressed in terms of the Krylov subspaces generated by the preconditioners 
$P_j^{-1}$ and right hand side $b$: 
\[ \bm{r}_{m} (\sigma)  \in  b + \mc{K}_{m}(P_{1}^{-1},P_{1}^{-1}\bm{b}) + \cdots + \mc{K}_{m}(P_{n_p}^{-1},P_{n_p}^{-1}\bm{b}). \]
Thus, at each iteration, the cross-product terms do not add to the search space and the dimension of the search space grows linearly
in the number
of preconditioners.
\end{remark}

\subsection{Implementation details}
As a result of Theorem \ref{gen.theo}, the MPGMRES-Sh search space can be built more efficiently than the original complete \break MPGMRES method. Although both methods compute the same search space, the complete MPGMRES  method as described in Algorithm~\ref{alg:mpgmressh} builds a search space $\mc{Z}_m$ with 
$\frac{n_p^m-n_p}{n_p-1}$ columns. However, by Theorem~\ref{gen.theo}, 
$\text{dim}\left(\Span\{\mc{Z}_m\}\right) = m n_p$, i.e., the search space has only linear growth 
in the number of iterations and number of preconditioners. 
Thus, an appropriate implementation of complete MPGMRES-Sh applied to our
case would include deflating the linearly dependent search space; 
for example, using a strong rank-revealing QR factorization as was done in~\cite{greifmpgmres}. 
This process would entail unnecessary computations, namely, in
the number of applications of the preconditioners, as 
well as in the orthonormalization and deflation. 

These additional costs can be avoided by directly building the linearly growing search space. 
Implementing this strategy only requires appending to the matrix of search directions
$\mc{Z}_{k-1}$, $n_p$ orthogonal columns spanning  ${\ran}(Z^{(k)})$, the range of
$Z^{(k)}$. Thus, only
$n_p$ independent searh directions are needed. It follows from Theorem~\ref{gen.theo} that 
any set of $n_p$ independent vectors in  
${\ran}(\mc{Z}_{k}) \setminus {\ran}(\mc{Z}_{k-1})$ suffice. We generate these vectors by applying each
of the $n_p$ preconditioners to the last column of $V^{(k)}$.

In other words, we replace step 3 in Algorithm~\ref{alg:mpgmressh}
with 
\begin{equation} \label{step3prime}
\mbox{\footnotesize 3':} ~~~~~~~~~~~~~~ Z^{(k)} = [P_{1}^{-1}\hat{v}_k,\ldots,P_{n_p}^{-1}\hat{v}_k], \mbox{\rm ~ where ~}  \hat{v}_k =  V^{(k)}e_{n_p} ,
~~~~~~~~~~~~~~~~~~
\end{equation}
and, where $e_{n_p}\in \R^{ n_p\times 1}$ is the last column of 
the $n_p \times n_p $ identity matrix $I_{n_p} $. Note that this implies that in line~\ref{line:min} the approximate solutions have the form $\bm{x}_m(\sigma) = \mathcal{Z}_m\bm{y}_m(\sigma)$ and the corresponding residual minimization problem is
\begin{equation}\label{e_sel_min}
\Vert \bm{r}_{m}(\sigma) \Vert_{2} =  \min_{\bm{y}\in \R^{mn_p}} \Vert \beta\bm{e}_{1} - \bar{\mc{H}}_{m}(\sigma;T_{m})\bm{y} \Vert_{2}. 
\end{equation} 

\begin{remark}
Algorithm~\ref{alg:mpgmressh} with step 3' as above is essentially what is called {\em selective MPGMRES} in \cite{greifmpgmres}. The big difference here is that this ``selective" version
captures the whole original search space (by Theorem~\ref{gen.theo}), while in \cite{greifmpgmres} only
a subspace of the whole search space is chosen.
\end{remark}

This version of the algorithm is precisely the one we use in our numerical experiments.  
From \eqref{step3prime} it can be seen that we need $n_p$ solves with a preconditioner per iteration,
and that at the $k^{\text{th}}$ iteration, one needs to perform $\left(k - \frac{1}{2}\right) n_p^{2} + \frac{3}{2} n_p$ inner
products for the orthogonalization.
This is in contrast to FGMRES-Sh where only one solve with a preconditioner per iteration is needed
and only $k+1$ innere products at the $k^{\text{th}}$ iteration.
Nevertheless, as we shall see in sections~\ref{sec:Hydrology} and \ref{sec:MatFun}, MPGMRES-Sh
achieves convergence in fewer iterations and lower computational times. The storage and computational cost of
the MPGMRES-Sh algorithm is comparable to that of a block Krylov
method. The preconditioner solves can also be parallelized
across multiple cores, which further reduces the computational cost.

\subsection{Extension to more general shifted systems}
\label{sec:ExtensiontoK}
The proposed MPGMRES-Sh method, namely Algorithm~\ref{alg:mpgmressh} with step 3' as in \eqref{step3prime},
can be easily adapted for solving more general shifted systems of the form
\begin{equation}
(K+\sigma_{j} M)\bm{x}(\sigma) = \bm{b}, \quad j=1,\ldots,n_\sigma,
\label{eq:GeneralK}
\end{equation}
with shift-and-invert preconditioners of the form $P_{j}=(K+\tau_{j}M)^{-1}$, $j=1,\ldots, n_p$. 

Using the identity
\begin{equation}
(K+\sigma M)P_{\tau}^{-1}= (K + \sigma M)(K+\tau M)^{-1} = I + (\sigma-\tau)MP_{\tau}^{-1},
\label{eqn:generalidentity}
\end{equation}
the same multipreconditioned approach described for the case $M =I$
can be applied. The difference for this more general case is that the multipreconditioned search space is now based on $MP_{\tau}^{-1}$.

In this case, analogous to \eqref{eqn:Zk}, we have 
\begin{equation*} 
Z^{(k)} = [(K+\tau_1M)^{-1}V^{(k)},\ldots, (K+\tau_t M)^{-1}V^{(k)}] \in \mathbb{R}^{n\times n_p^{k} }, \quad V^{(k)} \in \mathbb{R}^{n\times n_p^{k-1}},
\end{equation*}
which is equivalently expressed as
\begin{equation} 
KZ^{(k)} + MZ^{(k)}T^{(k)}  =   V^{(k)}E^{(k)}.
\label{eqn:completempdM}
\end{equation}
Using \eqref{eqn:completempdM} and concatenating $Z^{(k)}$, $V^{(k)}$, $T^{(k)}$, and $E^{(k)}$ into matrices we obtain the matrix relation
\begin{equation}
K\mc{Z}_{m} + M\mc{Z}_{m}T_{m} = \mc{V}_{m}E_{m},
\label{eqn:appendixarnoldi}
\end{equation}
where $T_m$ and $E_m$ are defined as in \eqref{eqn:ArnoldiRelation}.

Note that by \eqref{eqn:generalidentity}, the multipreconditioned Arnoldi relationship holds, i.e., $M\mathcal{Z}_m = \mathcal{V}_{m+1}\bar{\mc{H}}_m$, which in conjunction with \eqref{eqn:appendixarnoldi} gives the general version of~\eqref{eqn:ArnoldiIdentity}:
\begin{equation*}
(K + \sigma M)\mathcal{Z}_m = \mathcal{V}_{m+1}{\left( \begin{bmatrix} E_m \\ 0 \end{bmatrix} + \bar{\mc{H}}_m(\sigma I - T_m) \right)} \equiv \mc{V}_{m+1}{ \bar{\mathcal{H}}_m(\sigma;T_m)}.
\end{equation*}
As before, the approximate solutions have the form $\bm{x}_m(\sigma) = \mathcal{Z}_m\bm{y}_m(\sigma)$ and the corresponding residual minimization problem is~\eqref{e_sel_min}. It follows from \eqref{eqn:generalidentity} that only a basis for the space $\Span\{M\mathcal{Z}_m\}$ needs to be computed. This is due to the shift-invariant property \[\Span\{M\mathcal{Z}_m \} = \Span\{(K+\sigma M)\mathcal{Z}_m\}.\]
Thus, all we need to do is to use preconditioners of the form $P_j^{-1} = (K+\tau_j M)^{-1}$ in the input to Algorithm~\ref{alg:mpgmressh} and replace step 4 in Algorithm~\ref{alg:mpgmressh} with
\begin{equation} \label{step4prime}
\mbox{\footnotesize 4':} ~~~~ ~~~~~~~~~	 ~~ ~~ ~ ~ W= MZ^{(k)} .
~~~~~~~~~~~~~~~~~~~~~~ ~~~~~~~~~~~~~~~~~~~
~~~~~~~~~~~~~~~~~~~~~~ ~~~~~~~~~~~~~~~~~~~
\end{equation}

We remark that the analysis and results of Theorem \ref{gen.theo} remain valid for these more general shifted systems. To see this, consider the transformation of \eqref{eq:GeneralK} to $(KM^{-1} + \sigma I)x_M(\sigma) = b$ by the change of variables $x_M(\sigma) = Mx(\sigma)$. Note that this transformation is only valid when $M$ is invertible. As a result, we have a linear system of the form~\eqref{eq:General} and all the previous results are applicable here. The analysis can be extended to the case when $M$ is not invertible; however, this is outside the scope of this paper and is not considered here. For completeness, the resulting algorithm is summarized in Algorithm~\ref{alg:mpgmressh_sel}, which can be found in Appendix~\ref{app_gen}. 

\subsection{Stopping criteria}\label{ss_stop} We discuss here the stopping criteria used to test the convergence of the solver. As was suggested by Paige and Saunders~\cite{paige1982lsqr}, we consider the stopping criterion 
\begin{equation}
\norm{r_m(\sigma)} \>\leq\> \text{btol}\cdot\norm{b} + \text{atol}\cdot\norm{A + \sigma I}\norm{x_m(\sigma)}.
\end{equation}
We found this criterion gave better performance on highly ill-conditioned matrices as compared to the standard stopping criterion obtained by setting atol $=0$. We can evaluate $\norm{x_m(\sigma)}$ using the relation $\norm{x_m(\sigma)} = \norm{\mathcal{Z}_my_m(\sigma)}$. While we do not have access to $\norm{A + \sigma I}$, we propose estimating it as follows
\begin{equation} \label{eqn:normest}
 \norm{A + \sigma I} \approx \max_{k=1,\dots,mn_p} |\lambda_k| \qquad \mc{H}_m(\sigma;T_m)z_k = \lambda_k \mc{H}_mz_k.
\end{equation}
The reasoning behind this approximation is that the solution of the generalized eigenvalue problem $\mc{H}_m(\sigma;T_m)z_k = \lambda_k \mc{H}_mz_k$ are Harmonic Ritz eigenpairs, i.e., they satisfy \[(A+\sigma I)u - \theta u \perp \text{Span}\{\mc{V}_m\}, \qquad u \in \text{Span}\{\mc{V}_{m+1}\bar{\mc{H}}_m\};\]  
 and therefore, can be considered to be approximate eigenvalues of $A+\sigma I$. The proof is a straightforward extension of~\cite[Proposition 1]{saibaba2013flexible}. Numerical experiments confirm that MPGMRES-Sh with the above stopping criterion does indeed converge to an acceptable solution and is especially beneficial for highly ill-conditioned matrices $A$. 

Another modification to the stopping criterion needs to be made to account for inexact preconditioner solves. Following the theory developed in~\cite{simoncini2003theory}, a simple modification was proposed in~\cite{saibaba2013flexible}. Computing the true residual would require an extra application of a matrix-vector product with $A$. Typically, this additional cost is avoided using the Arnoldi relationship and the orthogonality of $V_m$ to compute the residual. 
However, when inexact preconditioners are used, the computed search space is a perturbation of the desired search space and the residual can only be computed approximately. Assuming that the application of the inexact preconditioners has a relative accuracy $\varepsilon$, based on~\cite{saibaba2013flexible}, we use the modified stopping criterion
\begin{equation}
\norm{r_m(\sigma)} \leq \text{btol}\cdot\norm{b} + \text{atol}\cdot\norm{A + \sigma I}\norm{x_m(\sigma)} + \varepsilon\cdot\| y_m(\sigma)\|_1,
\end{equation}
where $y_m(\sigma)$ is the solution of \eqref{eq:minsigma}, i.e., step 12 of our algorithm, 
namely Algorithm~\ref{alg:mpgmressh} with step 3' as in~\eqref{step3prime} (or Algorithm~\ref{alg:mpgmressh_sel} with $K =A$ and $M=I$).

\section{Application to Hydrology}
\label{sec:Hydrology}
\subsection{Background and motivation}

Imaging the subsurface of the earth is an important challenge in many
hydrological applications such as groundwater remediation and the
location of natural resources. Oscillatory hydraulic tomography
is a method of imaging that uses periodic pumping tests to
estimate important aquifer parameters, such as specific storage
and conductivity. Periodic pumping signals are imposed at
pumping wells and the transmitted effects are measured at
several observation locations. The collected data is then used
to yield a reconstruction of the hydrogeological parameters of
interest. The inverse problem can be tackled using the geostatistical approach; for details of this
application see~\cite{saibaba2013flexible}. 
A major challenge in solving these inverse problems
using the geostatistical approach is the cost of constructing the Jacobian, which represents the sensitivity
of the measurements to the unknown parameters.  In~\cite{saibaba2013flexible}, it is shown that constructing the Jacobian  requires repeated solution to a sequence of shifted systems.  An efficient solver for the forward problem can drastically reduce the overall computational
time required to solve the resulting inverse problem. In this section, we
demonstrate the performance of MPGMRES-Sh for solving the forward
problem with a periodic pumping source.

The equations governing groundwater flow through an aquifer for
a given domain $\Omega$ with boundary $\partial \Omega =
\partial \Omega_D \cup \partial \Omega_N \cup \Omega_W$
are given by
\begin{align} \label{eqn:timedomain}
S_s(x) \frac{\partial \phi(x,t)}{\partial t} - \nabla \cdot
\left(K(x) \nabla \phi(x,t)\right) & = q(x,t) & x &\in \Omega,
\\ \nonumber
\phi(x,t) & = 0 & x & \in \partial \Omega_D,  \\ \nonumber 
\nabla \phi(x,t) \cdot n & = 0 & x &\in \partial \Omega_N \\
\nonumber
K(x) \left(\nabla \phi(x,t)\cdot n\right) & = S_y \frac{\partial
\phi(x,t)}{\partial t} & x &\in \partial \Omega_W \\ \nonumber
\end{align}
where $S_s(x)$ (units of L$^{-1}$) represents the specific storage, $S_y$
(dimensionless) represents the specific yield and $K (x)$ (units of $L/T$) represents
the hydraulic conductivity. The boundaries $\partial \Omega_D,
\partial \Omega_N$, and $\partial \Omega_W$ represent Dirichlet,
Neumann, and the linearized water table boundaries, respectively.
In the case of one source oscillating at a fixed frequency
$\omega$ (units of radians/T) , $q(x,t)$ is given by
\begin{equation} \label{eq:oscillation} \nonumber
 q(x,t) = Q_0\delta(x-x_s) \cos(\omega t) .
 \end{equation}
We assume the source to be a point source oscillating at a known
frequency $\omega$ and peak amplitude $Q_0$ at the source
location $x_s$. Since the solution is linear in time, we assume
the solution (after some initial time has passed) can be
represented as
\begin{equation} \label{eqn:measurementequation} \nonumber
 \phi(x,t) = \Re(\Phi(x) \exp(i\omega t) ),
\end{equation}
where $\Re(\cdot)$ is the real part and $\Phi(x)$, known as
the phasor, is a function of space only that contains information
about the phase and amplitude of the signal. Assuming this form of the
solution, the equations~\eqref{eqn:timedomain} in the so-called phasor
domain are
\begin{align} \label{eqn:phasor1}
- \nabla \cdot \left(K(x) \nabla \Phi(x)\right) + i\omega S_s(x)
\Phi(x) = & \quad Q_0\delta(x-x_s) &\qquad x &\in \Omega, \\
\nonumber
\Phi(x) = & \quad 0 &\quad x &\in \partial \Omega_D, \\
\nonumber
\nabla \Phi(x) \cdot n = & \quad 0 &\qquad x &\in \partial
\Omega_N, \\ \nonumber
K(x) \left(\nabla \Phi(x) \cdot n\right) = &\quad i \omega S_y
\Phi(x) &\qquad x &\in \partial \Omega_W.
\end{align}
The differential equation~\eqref{eqn:phasor1} along with the
boundary conditions are discretized using standard finite elements implemented through the FEniCS software package~\cite{LoggMardalEtAl2012a, LoggWells2010a, LoggWellsEtAl2012a}. Solving the discretized equations for several
frequencies results in solving systems of shifted equations of the form
\begin{equation}
 \label{eqn:genshifted}
\left( K + \sigma_j M \right) x_j  = b \qquad j=1,\dots,n_\sigma,
\end{equation}
where, $K$ and $M$ are the stiffness and mass matrices,
respectively. The relevant MPGMRES-Sh algorithm for
this system of equations 
is the one  described in 
Section~\ref{sec:ExtensiontoK}, and summarized in Algorithm~\ref{alg:mpgmressh_sel}.

\subsection{Numerical examples}
\label{numer:sec}

In this section, we consider two synthetic aquifers in our test
problems. In both examples, the equations are discretized using
standard linear finite elements implemented using FEniCS
and Python as the user-interface. The first example is a two-dimensional depth-averaged aquifer in a
rectangular domain discretized using a regular grid. The second
test problem is a challenging $3$-D synthetic example
chosen to reflect the properties observed at the Boise
Hydrogeological Research Site (BHRS) \cite{cardiff2013hydraulic}; see Figure
\ref{fig:boise}. Although the domain is regular, modeling the
pumping well accurately requires the use of an unstructured grid. The numerical experiments were
 all performed on an HP workstation with 16 core Intel Xeon E5-2687W (3.1 GHZ) processor running Ubuntu 14.04. The machine has $128$ GB RAM and 1 TB hard disk space.

\begin{figure}[!ht]
\centering
\includegraphics[scale = 0.37]{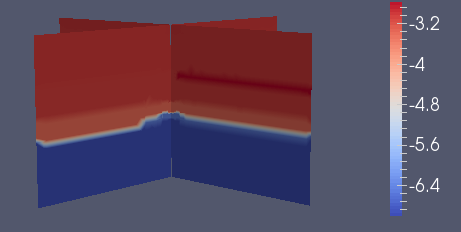}
\qquad
\includegraphics[scale =
0.269]{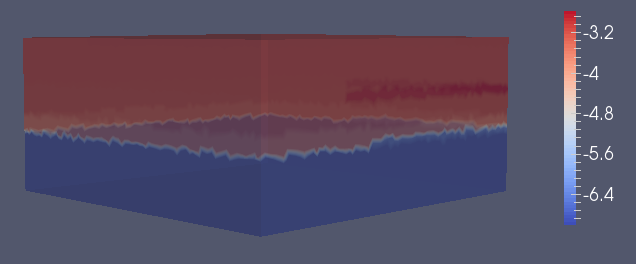}
\caption{Two views of the log conductivity field of the synthetic example. The
pumping well is located at the center of the domain.}
\label{fig:boise}
\end{figure}

\begin{figure}[!ht]
\centering
\includegraphics[scale = 0.4]{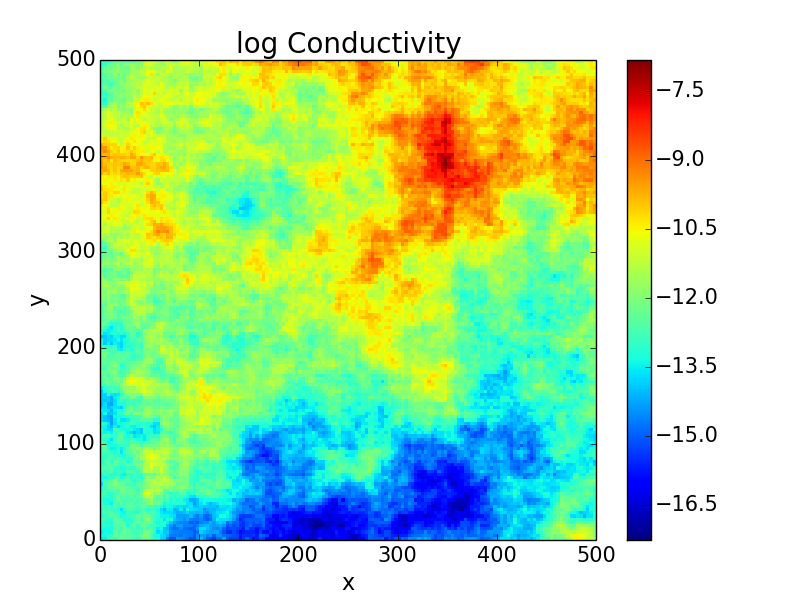}
\caption{A realization of the log conductivity field drawn from
$\mathcal{N}(\mu(x), \kappa(x,y))$ used in Test Problem 1. The parameter $\mu$ is specified in Table~\ref{tab:parameters} and $\kappa(x,y)$ is given by~\eqref{eqn:kernel}. The dimension of the system is $22801$.}
\label{fig:twodcond}
\end{figure}

\subsubsection{Test Problem 1} We consider a two-dimensional
aquifer in a rectangular domain with Dirichlet boundary
conditions on all boundaries. The log conductivity is chosen to
be a random field generated from the Gaussian process
$\mathcal{N}(\mu(x), \kappa(x,y))$, where $\kappa(\cdot,\cdot)$
is an exponential kernel taking the form
\begin{equation}\label{eqn:kernel} \kappa(x,y) = 4 \exp \left( -
\frac{2\normtwo{ x - y}}{L}\right), \end{equation} with $\mu(x)$ the mean of the log-conductivity chosen to be
constant and $L$ the length of the aquifer domain. Other
parameters for the model are summarized in Table
\ref{tab:parameters}; see Figure \ref{fig:twodcond} for the particular
realization of the conductivity field. As mentioned earlier, the
bottleneck for the large-scale reconstruction of parameters is
the repeated solution of the shifted system of
equations~\eqref{eqn:genshifted}. To choose realistic parameters
following~\cite{saibaba2013flexible}, $200$ frequencies evenly spaced in the range $\omega
\in [\frac{2\pi}{600},\frac{2\pi}{3}]$ were considered, which results in $200$
shifted systems of dimension $22801$.

\begin{table}[h] 
 \centering
\caption{Parameters chosen for test
problem}\label{tab:parameters}
\begin{tabular}{llc}

Definition &  Parameters  & Values \\
 \hline 
Aquifer length&  L (m) & 500 \\
Specific storage &  $\log S_s$ (m$^{-1}$) &  $-11.52$\\
Mean conductivity  &  $\mu(\log K)$ (m/s) & $ -11.52$ \\
Variance of conductivity   & $\sigma^2(\log K) $ & $2.79$  \\
Frequency range &$\omega$ ($s^{-1}$) &
$[\frac{2\pi}{600},\frac{2\pi}{3}]$ \\
\end{tabular}
\end{table}

We compare the MPGMRES-Sh algorithm proposed here with several
other solvers. We consider the `Direct' approach, which
corresponds to solving the system for each frequency using a
direct solver. Since this requires factorizing a matrix for each
frequency, this is expensive. `GMRES-Sh' refers to using
preconditioned GMRES to solve the shifted system of equations
using a single shift-and-invert preconditioner, with the frequency chosen
to be the midpoint of the range of frequencies. While this is
not the optimal choice of preconditioner, it is a representative example for illustrating that one preconditioner does not effectively precondition all systems. We also provide a 
comparison with `FGMRES-Sh'. Following~\cite{saibaba2013flexible}, we choose preconditioners of
the form $K + \tau_k M$ for $k = 1, ..., m$, where $m$ is the
maximum dimension of the Arnoldi iteration. Let $\bar{\tau} = \{
\bar{\tau}_1, \dots, \bar{\tau}_{n_p}\}$ be the list of values
that $\tau_k$ can take with $n_p$ denoting the number of
distinct preconditioners used. In \cite{saibaba2013flexible}, it is shown that systems with frequencies nearer to the origin converge
slower, so we choose the values of $\bar{\tau}$ to be evenly spaced in a log
scale in the frequency range $\omega \in
[\frac{2\pi}{600},\frac{2\pi}{3}]$. For FGMRES-Sh with
$n_p$ preconditioners we assign the first $m/n_p$ values of
$\tau_k$ as $\bar{\tau}_1$, the next $m/n_p$ values of $\tau_k$
to $\bar{\tau}_2$ and so on. If the algorithm has not converged
in $m$ iterations, we cycle over the same set of preconditioners. `MPGMRES-Sh' uses the same
preconditioners as `FGMRES-Sh' but builds a different search
space. We set the size of the bases per preconditioner to be $5$, i.e., $m/n_p = 5$. A relative residual of less than $10^{-10}$ was chosen as
the stopping criterion for all the iterative solvers, i.e., atol was set to~$0$.
Furthermore, all the preconditioner solves were performed using
direct solvers and can be treated as exact in the absence of
round-off errors.

\begin{figure}[!htbp]
\centering
\includegraphics[scale = 0.4]{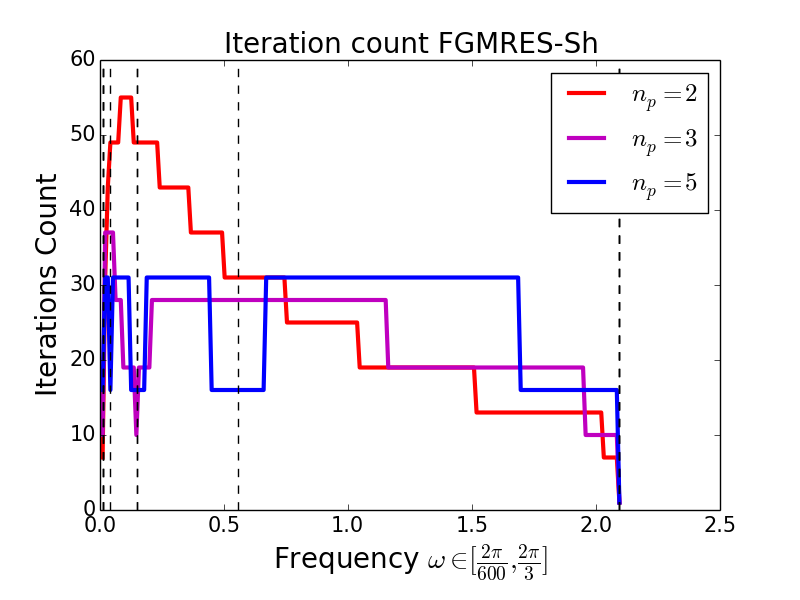}\\
\includegraphics[scale = 0.4]{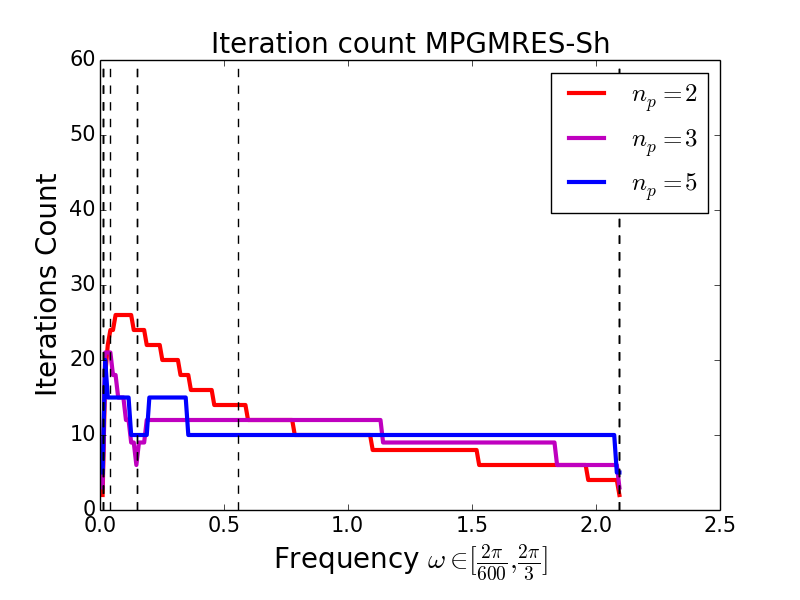}\\
\includegraphics[scale = 0.4]{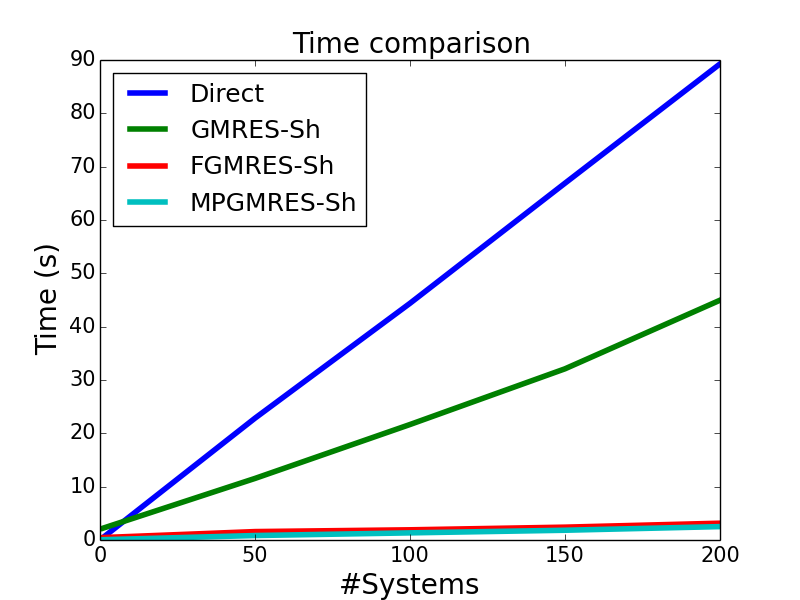}
\caption{  Iteration counts comparing  FGMRES-Sh (top) to MPGMRES-Sh (center) across the frequency range - the number of preconditioners are varied from $2$ to $5$. The system size is $22801$. (bottom) Time comparison (in seconds) between different solvers - Direct, GMRES-Sh, FGMRES-Sh and MPGMRES-Sh (Algorithm~\ref{alg:mpgmressh_sel}).} 
\label{fig:twodresults}
\end{figure}

We report the results of the various aforementioned solvers
in Figure~\ref{fig:twodresults}.  In the top two plots, we
compare  MPGMRES-Sh with FGMRES-Sh and we
see that the iteration count to reach the desired tolerance is
less with MPGMRES-Sh than with the use of FGMRES-Sh. The
vertical lines on the plot denote the preconditioners chosen in
the case when $n_p = 5$. In the bottom plot,
we compare the solution times (in seconds) of the solvers to solve the $200$
shifted systems. We observe that the `Direct' method is
computationally prohibitive because the growth is linear with the
number of shifted systems. GMRES-Sh uses the shift-invariance
property but the convergence with each frequency is different
and the number of iterations for some frequencies can be quite
large. Both FMRES-Sh and MPGMRES-Sh appear to behave independently
of the number of shifted systems to be solved and MPGMRES-Sh
converges marginally faster.

\subsubsection{Test Problem 2}\label{sec:real} We now compare
the solution time for a realistic 3-D problem on an unconfined aquifer. The aquifer is of  size $60 \times 60 \times 27$ m$^3$ and the
pumping well is located in the center with the pumping occurring
over a $1$m interval starting $2$m below the water table. For
the free surface boundary (top), we use the linearized water table
boundary $\partial \Omega_W$;  for the other boundaries we use a no-flow boundary
for the bottom surface $\partial \Omega_N$ and Dirichlet boundary conditions for the
side walls $\partial\Omega_D$. The dimension of the sparse matrices $K$ and $M$ is $132089$ with
approximately $1.9$ million nonzero entries each (both matrices have the same sparsity pattern). The number of
shifts $\sigma_j$ chosen is $100$ with the periods evenly spaced in range
from $10$ seconds to $15$ minutes, resulting in $\omega \in [\frac{2\pi}{900}, \frac{2\pi}{10}]$.
Since the `Direct' and `GMRES-Sh' approaches are infeasible on
problems of this size, we do not provide comparisons with them.
We focus on studying the iteration counts and time taken by
FGMRES-Sh and MPGMRES-Sh. The choice of preconditioners and the
construction of basis for FGMRES-Sh and MPGMRES-Sh are the same
as that for the previous problem. The number of preconditioners
$n_p$ is varied in the range $2$--$5$. However, one crucial difference is
that the preconditioner solves are now done using an iterative method,
specifically,  preconditioned CG with an algebraic multigrid (AMG) solver as a preconditioner available
through the PyAMG package~\cite{bell2011pyamg}. Following~\cite{saibaba2013flexible}, the stopping
criterion used for the preconditioner solve required the
relative residual to be less than $10^{-12}$. The number of
iterations and the CPU time taken by FGMRES-Sh and MPGMRES-Sh
have been displayed in Table~\ref{table:iters}; as can be seen,
MPGMRES-Sh has the edge over FGMRES-Sh both in iteration counts
and CPU times. In Figure~\ref{fig:iters} we report on the number of iterations for
the different shifted systems.

\begin{table}[!ht]
\begin{center}
\caption{Comparison of MPGMRES-Sh with FGMRES-Sh on Test Problem 2.\label{table:iters} }
\begin{tabular}{c|c c|c c}
\multirow{2}{*}{$n_p$} & \multicolumn{2}{|c|}{FGMRES-Sh} &
\multicolumn{2}{c}{MPGMRES-Sh} \\ \cline{2-5}
& Matvecs & CPU Time [s] & Matvecs & CPU Time [s]\\ \hline
$2$ & $58$ & $160.3$ & $36$ & $87.0$  \\ 
$3$ & $52$ & $130.3$ & $24$ & $52.1$ \\
$5$ & $44$&  $104.2$ & $20$ & $36.7$\\ 
\end{tabular}
\end{center}
\end{table}

\begin{figure}[!ht]
\centering
\includegraphics[scale = 0.45]{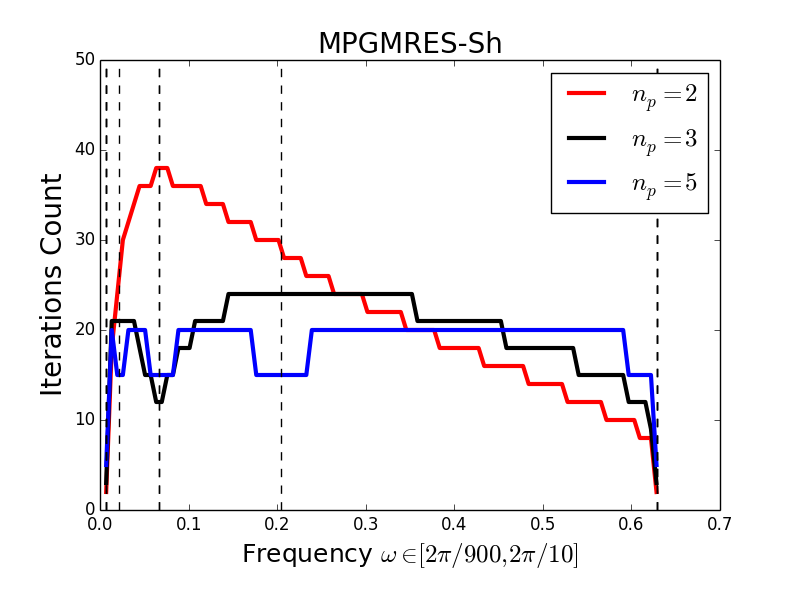}
\caption{Variation of iteration count with increasing number of
preconditioners in MPGMRES-Sh (Algorithm~\ref{alg:mpgmressh_sel}) for Test Problem 2.  The preconditioners are chosen on a log scale and their locations are highlighted in dashed lines. } 
\label{fig:iters}
\end{figure}

\section{Matrix Functions}
\label{sec:MatFun}
Matrix function evaluations are relevant in many applications;
for example, the evaluation of $\exp(-tA)b$ is important in the
time-integration of large-scale dynamical
systems~\cite{bakhos2015fast}. In the field of statistics and
uncertainty quantification, several computations involving a
symmetric positive definite covariance matrix $A$ can be
expressed in terms of matrix functions. For example, evaluation
of $A^{1/2}\xi$ can be used to sample from $\mathcal{N}(0,A)$
where $\xi \sim\mathcal{N}(0,I)$ as was implemented
in~\cite{saibaba2012efficient}, and an unbiased estimator to
$\log\left(\det(A)\right)$ can be constructed as
\[ \log\left(\det(A)\right) = \text{trace}\left(\log(A)\right) \approx
\frac{1}{n_s}\sum_{k=1}^{n_s} \zeta_k^T\log(A)\zeta_k,\]
where $\zeta_k$ is drawn i.i.d.\ from a Rademacher or Gaussian
distribution~\cite{saibaba2015fast}. The evaluation of the
matrix function can be carried out by representing the function
as a contour integral
\begin{equation*}
f(A) = \frac{1}{2\pi i}\int_\Gamma  {f(z)}(zI-A)^{-1} dz ,
\end{equation*}
where $\Gamma$ is a closed contour lying in the region of
analyticity of $f$. The matrix $A$ is assumed to be positive
definite and we consider functions $f$ that are analytic except
for singularities or a branch cut on (or near) $(-\infty,0]$. We
consider the approach in~\cite{hale2008computing} that uses a
conformal map combined with the trapezoidal rule to achieve an
exponentially convergent quadrature rule as the number of
quadrature nodes $N\rightarrow \infty$. The evaluation of the
matrix function $f(A)b$ can then be approximated by the sum
\begin{equation}\label{eqn:trapz}
f(A)b \quad \approx \quad f_N(A)b \quad = \quad \sum_{j=1}^{N}
w_j (z_jI - A)^{-1}b,
\end{equation}
where $w_k$ and $z_k$ are quadrature weights and nodes. The
convergence of the approximation as $N\rightarrow \infty$ is
given by an expression of the type
\[ \norm{f(A)-f_N(A)} \> = \> \mathcal{O}\left( e^{-C_1\pi^2
N/(\log(M/m) + C_2)} \right),\]
where $C_1$ and $C_2$ are two constants depending on the
particular method used and $m,M$ are the smallest and largest
eigenvalue of $A$, respectively~\cite{hale2008computing}. Given a
tolerance $\epsilon$, we choose $N$ according to the formula
\[ N = \left\lceil\frac{(\log(M/m) + C_2)}{C_1\pi^2 \epsilon}
\right\rceil.\]
In our application, we choose $\epsilon = 10^{-6}$, so that the
expression~\eqref{eqn:trapz} requires the solution
of shifted system of equations $(z_j I - A)x_j = b$ for
$j=1,\dots,N$ and can be computed efficiently using the
MPGMRES-Sh algorithm. Note that $N$ depends on the function $f(\cdot)$ and the condition number of the matrix $A$.

\begin{table}[!ht]
\centering
\caption{Comparison of number of iterations by the FGMRES-Sh and
MPGMRES-Sh solvers for the computation of $\exp(-A)b$, $\log(A)b$, and $A^{1/2}b$ evaluated
using~\eqref{eqn:trapz}. CPU times are reported in seconds. }
\begin{tabular}{cccc|cc} 
\multirow{2}{*}{Matrix name}& {Condition}
& \multicolumn{2}{c}{FGMRES-Sh} & \multicolumn{2}{c}{MPGMRES-Sh}
\\ \cline{3-6}
& number & Iter. count & CPU Time & Iter. count & CPU Time \\ \hline
\multicolumn{6}{c}{$\exp(-A)b$} \\\hline
plbuckle & $1.28 \times 10^ 6$ & 433 &  6.11 & 114 & 1.66 \\
nasa1824 & $1.89 \times 10^6$ &  529 & 9.20 & 120  & 2.82 \\
$1138$\_bus & $8.57 \times 10^6$ & 505 &  8.42 & 126 & 2.58 \\ \hline
\multicolumn{6}{c}{$\log(A)b$} \\\hline
plbuckle & $1.28 \times 10^ 6$ & 217 &  1.07 & 78 & 0.48 \\
nasa1824 & $1.89 \times 10^6$ &  217 & 1.65 & 81  & 0.85 \\
$1138$\_bus & $8.57 \times 10^6$ & 193  & 0.76 & 75 & 0.44 \\ \hline
\multicolumn{6}{c}{$A^{1/2}b$} \\ \hline
plbuckle & $1.28 \times 10^ 6$ & 64 &  0.43 & 60 & 0.19 \\
nasa1824 & $1.89 \times 10^6$ &  82 & 0.67 & 75  & 0.36 \\
$1138$\_bus & $8.57 \times 10^6$ &  82 & 0.38 & 69 & 0.23 \\ \hline
\end{tabular}•
\label{tab:matfun}
\end{table}


For these experiments we used the stopping criterion in Section~\ref{ss_stop} with atol = btol = $10^{-10}$. Three preconditioners were used with shifts at $z_1$, $z_{N/2}$, and $z_{N}$. For FGMRES-Sh, we use the same strategy to cycle through the preconditioners. Furthermore, $m/n_p$ was set to be $5$. 

We focus on evaluating the following three important matrix functions
$\exp(-A)$, $\log(A)$ and $A^{1/2}$, and in each case we evaluate
$f(A)b$ for a randomly generated vector of appropriate dimension.
For $\exp(-A)b$ evaluation, we use Method 1, whereas for $\log(A)b$ we use Method 2, and finally for 
$A^{1/2}$ we use Method 3 as described in~\cite{hale2008computing}. 
We take several matrices from the UF
sparse matrix collection~\cite{davis2011university}, which have previously been studied in the context of computing matrix functions in~\cite{chen2011computing}. The number of iterations taken by FGMRES-Sh
and MPGMRES-Sh is provided in
Table~\ref{tab:matfun}.  As can be seen, MPGMRES-Sh takes fewer
preconditioner solves (recall that the matrix-vector products with the shifted matrices are not necessary) to converge to
the desired tolerance. The CPU times reported are averaged over $5$ independent runs to get accurate timing results. In these examples, as in the hydrology examples, MPGMRES-Sh is more effective both in terms of iteration count and overall run time.

 \section{Conclusions}
\label{sec:Conc}

In this paper, we derived a new algorithm for efficiently solving shifted systems of equations as described by \eqref{eq:General} and \eqref{eq:GeneralK}. The newly proposed algorithm combines the flexible Krylov framework for shifted systems, developed in~\cite{saibaba2013flexible}, with multi-preconditioned GMRES, developed in~\cite{greifmpgmres}, and allows one to build a 
search space using multiple shift-and-invert preconditioners. We showed that the search space produced by our algorithm 
grows linearly with the number of preconditioners and number of iterations. The resulting algorithm converges in 
fewer iterations than related solvers GMRES-Sh and FGMRES-Sh, and although the cost per iteration of MPGMRES-Sh is higher, in our experience, 
the overall execution time is generally lower. 
The numerical examples drawn from applications in hydrology and matrix function evaluations demonstrate the superior 
performance in terms of iteration counts and overall run time of MPGMRES-Sh as compared to other standard 
preconditioned solvers that use a single preconditioner at each iteration.

\section{Acknowledgments}
The authors are greatly indebted to Michael Saunders for a careful and thorough reading of the paper. We would also like to thank Warren Barrash for his help with the hydrology examples.

%

\appendix
\section{Case of $n_p > 2$ preconditioners}\label{appendixA} 
In this appendix we prove Theorem~\ref{gen.theo}.
We first prove versions of Lemmas
\ref{lem:CrossProduct} and \ref{lem:CrossProductMN} for $n_p \geq 2$ preconditioners.
\begin{lemma}
Let $\{P_{j}^{-1}\}_{j=1}^{n_p}$ be shift-and-invert
preconditioners as defined in \eqref{eq:ShiftandInvert}, where
$\tau_{j}\neq\tau_{i}$ for $j\neq i$. Then for any vector $v$,
\[\left(\prod_{j=1}^{n_{p}}P_{j}^{-1}\right)\bm{v} \in
\Span\{P_{1}^{-1}\bm{v},\ldots,P_{n_{p}}^{-1}\bm{v}\}.\]
\label{lem:CPgen}
\end{lemma}
\emph{Proof.}
We use induction on $n_{p}$. We show that there exist unique constants $\{ \gamma_{j}\}_{j=1}^{n_{p}}$
so that
\begin{equation}
\left(\prod_{j=1}^{n_{p}}P_{j}^{-1}\right) =
\sum_{j=1}^{n_{p}}\gamma_{j}P_{j}^{-1}.
\label{eq:Lemma1gen}
\end{equation}
holds for any integer $n_p$.
The proof for $n_{p}=1$ is straightforward. Assume~\eqref{eq:Lemma1gen} holds for $n_{p} >1$. For $n_{p}+1$
\[ \left(\prod_{j=1}^{n_{p} +1}P_{j}^{-1}\right) = \left(\prod_{j=1}^{n_{p}}P_{j}^{-1}\right)P_{n_{p} +1}^{-1} =
\left(\sum_{j=1}^{n_{p}}\gamma_{j}P_{j}^{-1}\right)P_{n_{p}+1}^{-1}.\]
The last equality follows because of the induction hypothesis. Next, assuming that $\tau_{n_{p}+1} \neq \tau_j$ for $j=1,\dots,n_{p}$, 
we use Lemma~\ref{lem:CrossProductMN} to simplify the expression
\begin{align*}  \sum_{j=1}^{n_{p}}\gamma_{j}P_{j}^{-1}P_{n_{p}+1}^{-1} = & \sum_{j=1}^{n_{p}}\gamma_j\left(\gamma_1^{(j)} P_j^{-1} + \gamma_2^{(j)}P_{n_{p}+1}^{-1}\right)\\
= & \sum_{j=1}^{n_{p}}  \gamma_j\gamma_1^{(j)} P_j^{-1}+ \left(\sum_{j=1}^{n_{p}}\gamma_j\gamma_2^{(j)}\right) P_{n_{p} +1}^{-1} 
\equiv  \sum_{j=1}^{n_{p}+1} \gamma_j'P_{j}^{-1}. ~~~~\cvd\
\end{align*}

\begin{lemma}
Let $\{P_{j}^{-1}\}_{j=1}^{n_p}$ be shift-and-invert
preconditioners as defined in \eqref{eq:ShiftandInvert} where
$\tau_{j}\neq\tau_{i}$ for $j\neq i$. Then,
for all vectors $v$ and all integers $m$,
\[\left(\prod_{j=1}^{n_p} P_{j}^{-m}\right)\bm{v} \in
\sum_{j=1}^{n_p} \mc{K}_{m} (P_{j}^{-1},P_{j}^{-1}\bm{v}).\]
\label{lem:CPMNgen}
\end{lemma}
\begin{proof} 
We proceed by induction on $m$. The cases when
$m=1$ is true by Lemma~\ref{lem:CPgen}.
Assume that the Lemma holds for $m > 1$, i.e., there exist
coefficients such that
\begin{equation*}
\left(\prod_{k =1}^{n_p} P_{k}^{-m}\right)\bm{v} =
\sum_{j=1}^{m} \alpha_{1}^{(j)} P_{1}^{-j}\bm{v} + \cdots +
\sum_{j=1}^{m} \alpha_{n_p}^{(j)} P_{n_p}^{-j}\bm{v} .
\end{equation*}
Next, consider
\begin{align*}
\left(\prod_{k=1}^{n_p}P_{k}^{-(m+1)}\right) \bm{v}& =
\left(\prod_{k=1}^{n_p}P_{k}^{-1}\right)\left(\prod_{k=1}^{n_p}
P_{k}^{-m}\right)\bm{v} \\
&= \left(\prod_{k=1}^{n_p}P_{k}^{-1}\right)\left(
\sum_{j=1}^{m} \alpha_{1}^{(j)} P_{1}^{-j}\bm{v} + \cdots +
\sum_{j=1}^{m} \alpha_{n_p}^{(j)} P_{n_p}^{-j}\bm{v}\right) \\
&= \sum_{j=1}^{m} \alpha_{1}^{(j)} \left(\prod_{k=
1}^{n_p} P_{k}^{-1}\right) P_{1}^{-j}\bm{v} + \cdots +
\sum_{j=1}^{m} \alpha_{n_p}^{(j)} \left( \prod_{k=1}^{n_p} P_{k}^{-1}\right)P_{n_p}^{-j} \bm{v}.
\end{align*}
The first equality is a result of the commutativity of the
shift-and-invert preconditioners. The second equality follows
from the application of the induction hypothesis. Next, consider the last expression, which 
we rewrite as 
\begin{align*}
\sum_{i=1}^{n_p} \sum_{j=1}^{m}\alpha_i^{(j)}P_i^{-j} \left(\prod_{k=1}^{n_p} P_k^{-1}\right)v 
&= \sum_{i=1}^{n_p} \sum_{j=1}^{m}\alpha_i^{(j)}P_i^{-j} \left(\sum_{k=1}^{n_p} \gamma_k P_k^{-1}v \right) \\
&= \sum_{i=1}^{n_p} \sum_{j=1}^{m}    
\sum_{k=1}^{n_p} 
\alpha_i^{(j)}
\gamma_k 
P_i^{-j} 
P_k^{-1}
v, 
\end{align*}
where the first equality follows from an application of Lemma~\ref{lem:CPgen}.  
By repeated application of Lemma~\ref{lem:CrossProductMN}, 
the Lemma follows.
\end{proof}

{\em Proof of Theorem~\ref{gen.theo}.}
Recall that by \eqref{eq:mpresidual} the residual
can be expressed as a multivariate-polynomial in the
preconditioners. Due to Lemma \ref{lem:CPMNgen}, any cross-product term can be reduced to a linear combination of powers
only of $P_{j}^{-1}$. Gathering  like terms together, the
residual can be expressed as a sum of single-variate
polynomials in each preconditioner.  \cvd

\section{Algorithm for more general shifted systems}\label{app_gen}
For convenience, we provide the selective version of the algorithm for solving~\eqref{eq:GeneralK}. The special case for~\eqref{eq:General} can be obtained by setting $K=A$ and $M=I$.

\begin{algorithm}[htp]
\begin{algorithmic}[1]
\REQUIRE 
Matrices $K$ and $M$, right-hand side $\bm{b}$, preconditioners $\{K + \tau_j M\}_{j=1}^{n_p}$, shifts $\{\sigma_j\}_{j=1}^{n_\sigma}$, $\{\tau_j\}_{j=1}^{n_p}$ and number of iterations $m$.
\STATE Compute $\beta = \norm{\bm{b}}_{2}$ and $V^{(1)} =  \bm{b}/\beta$.
\FOR {$k=1,\dots,m$ }  
\STATE 
$\hat{v}_k =  V^{(k)}e_{n_p}$,  \mbox{\rm and }  
$Z^{(k)} = [P_{1}^{-1}\hat{v}_k,\ldots,P_{n_p}^{-1}\hat{v}_k]$. 
\STATE $W= MZ^{(k)}$ 
\FOR {j = 1,\dots,k}
\STATE $H^{(j,k)} = (V^{(j)})^TW$
\STATE $W = W - V^{(j)}H^{(j,k)}$
\ENDFOR
\STATE $W =  V^{(k+1)}H^{(k+1,k)}$ \  \COMMENT{thin QR factorization}
\ENDFOR
\FOR {$j=1,\dots,n_\sigma$}
\STATE Compute $\bm{y}_m(\sigma_{j}) = \argmin_\bm{y} \norm{\beta \bm{e}_1 - \bar{\mc{H}}_m(\sigma_j;T_m)\bm{y} }$ for each shift. 
\STATE $\bm{x}_{m}(\sigma_{j}) = \mc{Z}_{m}\bm{y}_{m}(\sigma_{j})$, where $\mc{Z}_{m} = [Z^{(1)}, \cdots, Z^{(m)}]$ 
\ENDFOR
\RETURN The approximate solution $x_m(\sigma_j)$ for $j=1,\dots,n_\sigma$. 
\end{algorithmic}
\caption{Selective MPGMRES-Sh for~\eqref{eq:GeneralK}}
\label{alg:mpgmressh_sel}
\end{algorithm}

\bibliographystyle{abbrv} \bibliography{mpgmres-sh}

\begin{thebibliography}{10}

\bibitem{Ahmad.Szyld.VanGijzen.12}
M.~I. Ahmad, D.~B. Szyld, and M.~B. van Gijzen.
\newblock Preconditioned multishift {BiCG} for $\mathcal{H}_2$-optimal model
  reduction.
\newblock Technical Report 12-06-15, Department of Mathematics, Temple
  University, June 2012.
\newblock Revised March 2013 and June 2015.

\bibitem{bakhos2015fast}
T.~Bakhos, A.~K. Saibaba, and P.~K. Kitanidis.
\newblock A fast algorithm for parabolic {PDE}-based inverse problems based on
  laplace transforms and flexible {K}rylov solvers.
\newblock {\em Journal of Computational Physics}, 299:940--954, 2015.

\bibitem{baumann2014nested}
M.~Baumann and M.~B. van Gijzen.
\newblock Nested {K}rylov methods for shifted linear systems.
\newblock {\em SIAM Journal on Scientific Computing}, 37(5):S90--S112, 2015.

\bibitem{bell2011pyamg}
W.~Bell, L.~Olson, and J.~Schroder.
\newblock {PyAMG}: Algebraic multigrid solvers in {P}ython v2. 0, 2011.
\newblock {\em URL http://www. pyamg. org. Release}, 2, 2011.

\bibitem{cardiff2013hydraulic}
M.~Cardiff, W.~Barrash, and P.~K. Kitanidis.
\newblock Hydraulic conductivity imaging from 3-d transient hydraulic
  tomography at several pumping/observation densities.
\newblock {\em Water Resources Research}, 49(11):7311--7326, 2013.

\bibitem{chen2011computing}
J.~Chen, M.~Anitescu, and Y.~Saad.
\newblock Computing f(${A}$)b via least squares polynomial approximations.
\newblock {\em SIAM Journal on Scientific Computing}, 33:195--222, 2011.

\bibitem{davis2011university}
T.~A. Davis and Y.~Hu.
\newblock The {U}niversity of {F}lorida sparse matrix collection.
\newblock {\em ACM Transactions on Mathematical Software}, 38:1, 2011.

\bibitem{frommer1995many}
A.~Frommer, B.~N{\"o}ckel, S.~G{\"u}sken, T.~Lippert, and K.~Schilling.
\newblock Many masses on one stroke: Economic computation of quark propagators.
\newblock {\em International Journal of Modern Physics C}, 6:627--638, 1995.

\bibitem{greifmpgmres}
C.~Greif, T.~Rees, and D.~B. Szyld.
\newblock {MPGMRES}: a generalized minimum residual method with multiple
  preconditioners.
\newblock Technical Report 11-12-23, Department of Mathematics, Temple
  University, Dec. 2011.
\newblock Revised September 2012, January 2014, and March 2015. Also available
  as Technical Report TR-2011-12, Department of Computer Science, University of
  British Columbia.

\bibitem{Greif.Rees.Szyld.14a}
C.~Greif, T.~Rees, and D.~B. Szyld.
\newblock Additive {S}chwarz with variable weights.
\newblock In J.~Erhel, M.~Gander, L.~Halpern, G.~Pichot, T.~Sassi, and
  O.~Widlund, editors, {\em Domain Decomposition Methods in Science and
  Engineering XXI}, Lecture Notes in Computer Science and Engineering, pages
  661--668. Springer, Berlin and Heidelberg, 2014.

\bibitem{gu2007flexible}
G.-D. Gu, X.-L. Zhou, and L.~Lin.
\newblock A flexible preconditoned {A}rnoldi method for shifted linear systems.
\newblock {\em Journal of Computational Mathematics}, 25, 2007.

\bibitem{hale2008computing}
N.~Hale, N.~J. Higham, and L.~N. Trefethen.
\newblock Computing ${A}^\alpha$, $\log({A})$, and related matrix functions by
  contour integrals.
\newblock {\em SIAM Journal on Numerical Analysis}, 46:2505--2523, 2008.

\bibitem{higham2010computing}
N.~J. Higham and A.~H. Al-Mohy.
\newblock Computing matrix functions.
\newblock {\em Acta Numerica}, 19:159--208, 2010.

\bibitem{jegerlehner1996krylov}
B.~Jegerlehner.
\newblock {K}rylov space solvers for shifted linear systems.
\newblock {\em arXiv preprint hep-lat/9612014}, 1996.

\bibitem{LoggMardalEtAl2012a}
A.~Logg, K.-A. Mardal, G.~N. Wells, et~al.
\newblock {\em Automated Solution of Differential Equations by the Finite
  Element Method}.
\newblock Springer, 2012.

\bibitem{LoggWells2010a}
A.~Logg and G.~N. Wells.
\newblock {DOLFIN}: Automated finite element computing.
\newblock {\em ACM Transactions on Mathematical Software}, 37, 2010.

\bibitem{LoggWellsEtAl2012a}
A.~Logg, G.~N. Wells, and J.~Hake.
\newblock {\em {DOLFIN}: a C++/Python Finite Element Library}, chapter~10.
\newblock Springer, 2012.

\bibitem{meerbergen2003solution}
K.~Meerbergen.
\newblock The solution of parametrized symmetric linear systems.
\newblock {\em SIAM Journal on Matrix Analysis and Applications},
  24:1038--1059, 2003.

\bibitem{meerbergen2010lanczos}
K.~Meerbergen and Z.~Bai.
\newblock The {L}anczos method for parameterized symmetric linear systems with
  multiple right-hand sides.
\newblock {\em SIAM Journal on Matrix Analysis and Applications},
  31:1642--1662, 2010.

\bibitem{paige1982lsqr}
C.~C. Paige and M.~A. Saunders.
\newblock {LSQR}: An algorithm for sparse linear equations and sparse least
  squares.
\newblock {\em ACM Transactions on Mathematical Software}, 8:43--71, 1982.

\bibitem{saad1993flexible}
Y.~Saad.
\newblock A flexible inner-outer preconditioned {GMRES} algorithm.
\newblock {\em SIAM Journal on Scientific Computing}, 14:461--469, 1993.

\bibitem{saad1986gmres}
Y.~Saad and M.~H. Schultz.
\newblock {GMRES}: A generalized minimal residual algorithm for solving
  nonsymmetric linear systems.
\newblock {\em SIAM Journal on Scientific and Statistical Computing},
  7:856--869, 1986.

\bibitem{saibaba2013flexible}
A.~K. Saibaba, T.~Bakhos, and P.~K. Kitanidis.
\newblock A flexible {K}rylov solver for shifted systems with application to
  oscillatory hydraulic tomography.
\newblock {\em SIAM Journal on Scientific Computing}, 35:A3001--A3023, 2013.

\bibitem{saibaba2012efficient}
A.~K. Saibaba and P.~K. Kitanidis.
\newblock Efficient methods for large-scale linear inversion using a
  geostatistical approach.
\newblock {\em Water Resources Research}, 48, 2012.

\bibitem{saibaba2015fast}
A.~K. Saibaba and P.~K. Kitanidis.
\newblock Fast computation of uncertainty quantification measures in the
  geostatistical approach to solve inverse problems.
\newblock {\em Advances in Water Resources}, 82:124 -- 138, 2015.

\bibitem{simoncini2003theory}
V.~Simoncini and D.~B. Szyld.
\newblock Theory of inexact {K}rylov subspace methods and applications to
  scientific computing.
\newblock {\em SIAM Journal on Scientific Computing}, 25:454--477, 2003.

\bibitem{Simoncini.Szyld.07}
V.~Simoncini and D.~B. Szyld.
\newblock Recent computational developments in {K}rylov subspace methods for
  linear systems.
\newblock {\em Numerical Linear Algebra with Applications}, 14:1--59, 2007.

\bibitem{Szyld.Vogel.01}
D.~B. Szyld and J.~A. Vogel.
\newblock A flexible quasi-minimal residual method with inexact
  preconditioning.
\newblock {\em SIAM Journal on Scientific Computing}, 23:363--380, 2001.

\bibitem{vogel2007flexible}
J.~A. Vogel.
\newblock Flexible {BiCG} and flexible {Bi-CGSTAB} for nonsymmetric linear
  systems.
\newblock {\em Applied Mathematics and Computation}, 188:226--233, 2007.

\end{thebibliography}
\end{document}